\documentclass[12pt]{amsart}
\usepackage{amssymb}
\usepackage{amsthm}
\usepackage{tensor}

\theoremstyle{Theorem}
\newtheorem{thm}{Theorem}[section]
\newtheorem{lem}[thm]{Lemma}

\newtheorem{prop}[thm]{Proposition}
\newtheorem{cor}[thm]{Corollary}

\theoremstyle{remark}

\newtheorem{ex}[thm]{Example}

\theoremstyle{definition}
\newtheorem{dfn}[thm]{Definition}

\usepackage{amsmath}

\begin{document}

\title
{Cuntz-Pimsner Algebras and Twisted Tensor Products}
\author{Adam Morgan}
\address{School of Mathematical and Statistical Sciences
\\Arizona State University
\\Tempe, Arizona 85287}
\email{anmorgan@asu.edu}
\date{\today}
\maketitle

\begin{abstract}
Given two correspondences $X$ and $Y$ and a discrete group $G$ which acts
on $X$ and coacts on $Y$, one can define a twisted tensor product $X\boxtimes
Y$ which simultaneously generalizes ordinary tensor products and crossed
products by group actions and coactions. We show that, under suitable conditions,
the Cuntz-Pimsner algebra of this product, $\mathcal O_{X\boxtimes Y}$, is
isomorphic to a ``balanced'' twisted tensor product 
$\mathcal O_X\boxtimes_\mathbb T\mathcal O_Y$ of the Cuntz-Pimsner algebras
of the original correspondences. We interpret this result in several
contexts and connect it to existing results on Cuntz-Pimsner algebras
of crossed products and tensor products.
\end{abstract}

\section{Introduction}
In the author's previous paper \cite{Morgan}, 
it was shown that, under suitable conditions,
if $X$ and $Y$ are $C^*$-correspondences over $C^*$-algebras $A$ and $B$,
the Cuntz-Pimsner algebra $\mathcal O_{X\otimes Y}$ is isomorphic to
a subalgebra $\mathcal O_X\otimes_\mathbb T \mathcal O_Y$ of
$\mathcal O_X\otimes\mathcal O_Y$. In the present paper, we will
extend this result from ordinary tensor products to a certain class of 
``twisted'' tensor products.

Many constructions in operator algebras may be thought of as
``twisted'' tensor products, for example: crossed products by actions or
coactions of groups, $\mathbb Z_2$-graded tensor products and so on. In 
\cite{Woronowicz}, a very general construction of a ``twisted tensor product''
is presented.
Their construction involves two quantum groups $\mathbb G=(S,\Delta_S)$ and
$\mathbb H=(T,\Delta_T)$, two coactions $(A,\mathbb G,\delta_A)$ and $(B,\mathbb
H,\delta_B)$, and a bicharacter $\chi\in\mathcal U(\widehat S\otimes\widehat
T)$. Given this information, they define a twisted tensor product $A\boxtimes
_\chi B$. They also show that if $X$ and $Y$ are correspondences over $A$
and $B$ with compatible coactions of $\mathbb G$ and $\mathbb H$,
there is a natural way of defining a correspondence 
$X\boxtimes_\chi Y$ over $A\boxtimes_\chi B$.  
In this paper, we will work with a special case of this general
construction which is general enough to be useful but simple enough to be
very tractable.

Specifically, we are interested in the case where, for some discrete group
$G$, $S= c_0(G)$,
$T=C^*_r(G)$ and $\chi=W^G\in\mathcal U\big(\widehat{C^*_r(G)}\otimes C^*_r(G)\big)$
is the reduced bicharacter of $C^*_r(G)$ viewed as a quantum group.
In this case, we may view the coaction of $c_0(G)$ as an action of $G$
on $A$ (or $X$), and we will be able to describe most of the algebraic
properties of $A\boxtimes_\chi B$ and $X\boxtimes_\chi Y$ entirely in terms of elementary tensors.

In this simplified setting, we will prove our main result:
if $J_{X\boxtimes_\chi Y}=J_X\boxtimes_\chi J_Y$ then
$\mathcal O_{X\boxtimes_\chi Y}\cong\mathcal O_X\boxtimes_\chi
\mathcal O_Y$ (where $J_X=\phi^{-1}\big(\mathcal K(X)\big)\cap
\big(\text{ker}(\phi)\big)^\perp$ is the Katsura ideal). We will
then apply this result to some specific examples.
\section{Preliminaries}
\subsection{Correspondences and Cuntz-Pimsner Algebras}
For a general reference on correspondences, we refer the reader to 
\cite{Lance}. For Cuntz-Pimsner algebras, we recommend \cite{Katsura1},
\cite{Katsura2} and the brief overview in \cite{GraphAlgebraBook}.
We will briefly recall some of the basic facts here.

Suppose $A$ is a $C^*$-algebra and $X$ is a right $A$-module.
We say that $X$ has an \emph{$A$-valued inner product} if there is 
a map 
$$X\times X\ni(x,y)\mapsto\langle x,y\rangle_A\in A$$
which is $A$-linear in the second variable and satisfies the following
\begin{enumerate}
\item
$\langle x,x\rangle_A\geq 0$ for all $x\in X$ with equality if and only if
$x=0$
\item
$\langle x, y\rangle^*_A=\langle y,x\rangle_A$ for all $x,y\in X$
\item
$\langle x,y\cdot a\rangle_A=\langle x,y\rangle_Aa$ for all $x,y\in X$ 
and $a\in A$.
\end{enumerate}

We can define the following norm on $X$:
$$\|x\|_A:=\|\langle x,x\rangle_A\|^{\frac 1 2}$$

If $X$ is complete under the norm $\|\cdot\|_A$ defined above,
$X$ is called a \emph{right Hilbert $A$-module}.
Note that if $A=\mathbb C$ then $X$ is just a Hilbert space and
we can think of Hilbert modules as generalized Hilbert spaces where the scalars
are elements of some $C^*$-algebra $A$. Sometimes we will write $(X,A)$ or
$X_A$ if we wish to emphasize $A$.

Let $A$ be a $C^*$-algebra and let 
$X$ be a right Hilbert $A$-module. Suppose $T:X\to X$ is an
$A$-module homomorphism. 
If there is an $A$-module homomorphism $T^*$ such that
$$\langle T^*x,y\rangle_A=\langle x, Ty\rangle_A$$
for all $x,y\in X$, then we call $T$ \emph{adjointable}
and we refer to $T^*$ as the \emph{adjoint} of $T$. 
The set of all adjointable
operators on $X$ with the operator norm is a $C^*$-algebra. We denote 
this algebra by $\mathcal L(X)$. Given $x,y\in X$ we define the operator
$\Theta_{x,y}$ as follows: $\Theta_{x,y}(z)=x\langle y,z\rangle_A$.
The closed linear span of all such operators is a subalgebra of $\mathcal
L(X)$ which we call the \emph{generalized compact operators}. It is denoted
by $\mathcal K(X)$.

Suppose $X$ is a
right Hilbert $B$-module. Suppose further that we have a homomorphism $\phi:A\to\mathcal
L(X)$ for some $C^*$-algebra $A$.
This is called a \emph{left action of $A$ by adjointable operators}.
 We call the triple $(A,X,B)$ a 
\emph{$C^*$-correspondence} or simply a \emph{correspondence}.
For $a\in A$ and $x\in X$, we will write $a\cdot x$ for $\phi(a)(x)$.
If $A=B$ we call this a \emph{correspondence over $A$} (or $B$).
We call the left-action \emph{injective} if $\phi$ is injective
and \emph{non-degenerate} if $\phi(A)X=X$. If $\phi(A)\subseteq
\mathcal K(X)$, we say that the left action is \emph{implemented
by compacts}.
We will sometimes write $_AX_B$ to indicate that $X$ in an $A-B$
correspondence. Given an $A_1-B_2$ correspondence $X$ and an
$A_2-B_2$ correspondence $Y$, a 
\emph{correspondence isomorphism} is a triple $(\varphi_A,\Phi,\varphi_B)$
where $\Phi$
is a linear isomorphism $\Phi:X\to Y$ and $\varphi_A:A_1\to A_2$ and $\varphi_B
:B_1\to B_2$ are isomorphisms of $C^*$-algebras
such that the left and right actions and the inner product are preserved
by the maps:
\begin{align*}
\Phi(ax)&=\varphi_A(a)\Phi(x)\\
\Phi(xb)&=\Phi(x)\varphi_B(b)\\
\big\langle\Phi(x),\Phi(x')\big\rangle^Y_B&=\varphi_B\Big(\big\langle x,x'\rangle^X_B
\Big)
\end{align*}
Where $\langle\cdot,\cdot\rangle_B^Y$ denotes the $B$-valued inner product
on $Y$ and $\langle\cdot,\cdot\rangle_B^X$ denotes the $B$-valued inner product
on $X$. It will be convenient to introduce the following definition:

\begin{dfn}\label{GeneratingSystem}
Let $X$ be an $A-B$ correspondence. A \emph{generating system} for
$X$ is a triple $(A^0,X^0,B^0)$ where $A^0\subseteq A$, $X^0\subseteq X$
and $B^0\subseteq B$ such that $\overline{\text{span}}(A^0)=A$,
$\overline{\text{span}}(X^0)=X$, and $\overline{\text{span}}(B^0)=B$
and such that for all $x\in X^0$ we have that
$ax, xb\in X^0$ for all $a\in A^0$ and $b\in B^0$. If $A=B$ and $A^0=B^0$
we will denote the generating system by $(X^0,A^0)$.
\end{dfn}
We will make frequent use of the following fact:
\begin{lem}\label{IsomLemma}
Let $X$ be an $A_1-B_1$ correspondence and $Y$ be a $A_2-B_2$ correspondence.
Suppose that $(A^0_1,X^0,B^0_1)$ and
$(A^0_2,Y^0,B^0_2)$ are generating
sets for $X$ and $Y$ respectively. Let $\varphi_A:A_1\to A_2$ and 
$\varphi_B:B_1\to B_2$ be isomorphisms. Suppose there is a bijection
$\Phi_0:X^0\to Y^0$,
which preserves the inner product, left and right actions,
and scalar multiplication. That is
\begin{align*}
\big\langle\Phi_0(x),\Phi_0(x')\big\rangle_B^Y&=\varphi_B\Big(\langle x,x'
\rangle_B^X\Big)
&\text{for all $x,x'\in X^0$}\\
\Phi_0(ax)&=\varphi_A(a)\Phi_0(x) & \text{for all $x\in X^0$ and $a\in A^0$}\\
\Phi_0(xb)&=\Phi_0(x)\varphi_ B(b) & \text{for all $x\in X^0$ and $b\in B^0$}\\
\Phi_0(cx)&=c\Phi_0(x) & \text{for all $x\in X^0$ and $c\in\mathbb C$}
\end{align*}
Then $\Phi_0$ extends linearly and continuously to a correspondence
isomorphism $\Phi:X\to Y$.
\end{lem}
\begin{proof}
Let $x\in\text{span}(X^0)$, then $x=\sum_ic_ix_i$ for some $x_i\in X^0$ and
$c_i\in \mathbb C$.
We define $\Phi(x)=\sum_ic_i\Phi_0(x_i)$. First, we must verify that $\Phi$
is well defined on $\text{span}(X_0)$. Suppose $\sum_{i=1}^nc_ix_i$ and $\sum_
{i=1}^md_ix'_i$
are both equal to $x\in X$ with $x_i,x'_i\in X^0$. Let $y=\sum_{i=1}^nc_i\Phi_0(x_i)$
and $y'=\sum_{i=1}^md_i\Phi_0(x'_i)$. Then 
\begin{align*}
\|y-y'\|^2&=\big\langle y-y',y-y'\rangle_{B'}\\
&=\Big\langle\sum_{i=1}^nc_i\Phi_0(x_i)-\sum_{i=1}^md_i\Phi_0(x'_i),
\sum_{i=1}^nc_i\Phi_0(x_i)-\sum_{i=1}^md_i\Phi_0(x'_i)\Big\rangle_{B'}\\
&=\sum_{i,j=1}^{n,n}\overline{c_i}c_j\langle\Phi_0(x_i),\Phi_0(x_j)
\rangle_{B'}+\sum_{i,j=1}^{m,m}\overline{d_i}d_j\langle\Phi_0(x'_i),
\Phi_0(x'_j)\rangle_{B'}\\
&\quad-\sum_{i,j=1}^{n,m}\overline{c_i}d_j\langle\Phi_0(x_i),\Phi_0(x'_j)
\rangle_{B'}-\sum_{i,j=1}^{m,n}\overline{d_i}c_j\langle\Phi_0(x'_i),\Phi_0(x_j)
\rangle_{B'}\\
&=\sum_{i,j=1}^{n,n}\overline{c_i}c_j\varphi_B\big(\langle x_i,x_j
\rangle_B\big)+\sum_{i,j=1}^{m,m}\overline{d_i}d_j\varphi_B\big(\langle
x'_i,x'_j\rangle_B\big)\\
&\quad-\sum_{i,j=1}^{n,m}\overline{c_i}d_j\varphi_B\big(\langle x_i,x'_j
\rangle_B\big)-\sum_{i,j=1}^{m,n}\overline{d_i}c_j\varphi_B\big(\langle
x'_i,x_j\rangle_B\big)\\
&=\varphi_B\Big(\Big\langle\sum_{i=1}^nc_ix_i-\sum_{i=1}^md_ix'_i,
\sum_{i=1}^nc_ix_i-\sum_{i=1}^md_ix'_i\Big\rangle_B\Big)\\
&=\varphi_B\big(\langle x-x,x-x\rangle_B\big)\\
&=0
\end{align*}
where we have used the fact that $\varphi$ is an isomorphism and thus
linear.

For $x,x'\in X^0$ we have
\begin{align*}
\big\langle\Phi(x),\Phi(x')\big\rangle_B^Y&=\Big\langle\sum_ic_i\Phi_0(x_i),
\sum_jc'_j\Phi_0(x'_j)\Big\rangle_B^Y\\
&=\sum_{i,j}\overline c_ic'_j\big\langle\Phi_0(x_i),\Phi_0(x'_j)\big\rangle_B^Y\\
&=\sum_{i,j}\overline c_ic'_j\varphi_B\Big(\langle x_i,x_j'\rangle_B^X\Big)\\
&=\varphi_B\Big(\Big\langle\sum_ic_ix_i,\sum_jc'_jx'_j\Big\rangle_B^X\Big)\\
&=\varphi_B\Big(\langle x,x'\rangle_B^X\Big)
\end{align*}
Therefore, $\Phi$ preserves the inner product on $\text{span}(X^0)$ and thus
also preserves the norm on $\text{span}(X^0)$ and so $\Phi$ is bounded and
can be extended continuously to $\overline{\text{span}}(X^0)=X$. Hence, for
any $z\in X$ we can approximate $z\approx\sum_i c_iz_i$ with $z_i\in X^0$. Thus
for any $a_0\in A^0$ we have
\begin{align*}
\Phi(a_0z)&\approx \Phi\Big(a_0\sum_ic_iz_i\Big)
=\sum_ic_i\Phi(a_0z_i)
=\varphi_A(a_0)\sum_ic_i\Phi(z_i)
\approx\varphi_A(a_0)\Phi(z)
\end{align*}
so $\Phi(a_0z)=\varphi_A(a_0)\Phi(z)$ and similarly $\Phi(zb_0)=\Phi(z)
\varphi_B(b_0)$ for $b_0\in
B^0$. For arbitrary $a\in A$ we may approximate $a\approx\sum_ic_ia_i$ with
$a_i\in A^0$ and we see that
$$\Phi(az)\approx\sum c_i\Phi(a_iz)=\varphi_A\Big(\sum_ic_ia_i\Big)\Phi(z)\approx
\varphi_A(a)\Phi(z)$$
So $\Phi(az)=\varphi_A(a)\Phi(z)$ for all $a\in A$ and similarly $\Phi(zb)=\Phi(z)
\varphi_B(b)$
for all $b\in B$.
We already know that $\Phi$ is injective since it preserves the norm, so
we only have to show that it is surjective. To see this note that  
$$\Phi(X)=\overline{\text{span}}\big(\Phi_0(X^0)\big)=\overline{\text{span}}
(Y^0)=Y$$
So we have established that $\Phi$ is a linear isomorphism from $X$ to $Y$
which preserves the left and right actions and the inner product, in other
words $\Phi$ is a correspondence isomorphism.
\end{proof}

Given a Hilbert module $(X,A)$, we define the $\emph{linking algebra}$
of $(X,A)$ to be the $C^*$-algebra $L(X):=\mathcal K(X\oplus A)$.
There are complimentary projections $p$ and $q$ in $M(L(X))$ such that $pL(X)p
\cong\mathcal K(X)$, $pL(X)q\cong X$, and $qL(X)q\cong A$. This gives $L(X)$
the following block matrix decompostion:
$$L(X)=\begin{bmatrix}
\mathcal K(X) & X\\
\overline X & A
\end{bmatrix}$$ with 
$p = \begin{bmatrix}1&0\\0&0\end{bmatrix}$
and $q = \begin{bmatrix}0&0\\0&1\end{bmatrix}$.
The benefit of using linking algebras is that the algebraic properties
of $X$ are encoded into the multiplicative structure of $L(X)$:
\begin{align}
\begin{bmatrix}0&x\\0&0\end{bmatrix}\cdot\begin{bmatrix}0&0\\0&a\end{bmatrix}
&=\begin{bmatrix}0&xa\\0&0\end{bmatrix}\\
\begin{bmatrix}0&x\\0&0\end{bmatrix}^*\cdot\begin{bmatrix}0&y\\0&0\end{bmatrix}
&=\begin{bmatrix}0&0\\0&\langle x,y\rangle_A\end{bmatrix}\\
\begin{bmatrix}0&x\\0&0\end{bmatrix}\cdot\begin{bmatrix}0&y\\0&0\end{bmatrix}^*
&=\begin{bmatrix}\Theta_{x,y}&0\\0&0\end{bmatrix}
\end{align}

Given a correspondence $X$ over $A$, a \emph{Toeplitz representation}
of $(X,A)$ in a $C^*$-algebra $B$ is a pair $(\psi,\pi)$
where $\psi:X\to B$ is a linear map and $\pi:A\to B$ is a homomorphism 
satisfying:
\begin{align*}
\psi(xa)&=\psi(x)\pi(a)\\
\pi\big(\langle x,y\rangle_A\big)&=\psi(x)^*\psi(y)\\
\psi(ax)&=\pi(a)\psi(x)
\end{align*} 
By $C^*(\psi,\pi)$ we shall mean the $C^*$-subalgebra of $B$ generated
by the images of $\psi$ and $\pi$ in $B$.
There is a unique (up to isomorphism) $C^*$-algebra $\mathcal T_X$, called
the \emph{Toeplitz algebra} of $X$, which is generated by a universal 
Toeplitz representation $(i_X,i_A)$.

Let $X^{\otimes n}$ denote the $n$-fold internal tensor product 
(see \cite{Lance} for information on internal tensor products) of $X$
with itself. By convention, we let $X^{\otimes 0}=A$. Let $(\psi,\pi)$ 
be a Toeplitz representation of $X$ in $B$. Define the map 
$\psi^n:X^{\otimes n}\to B$ for each $n\in \mathbb N$ as follows:
let $\psi^0=\pi$, $\psi^1=\psi$, and set $\psi^n(x\otimes y)=\psi(x)
\psi^{n-1}(y)$ (where $x\in X$ and $y\in 
X^{\otimes n-1}$) for each $n>1$ .

We summarize Proposition 2.7 of \cite{Katsura1} as follows:
Let  $(\psi,\pi)$ be
a Toeplitz representation of a correspondence $X$ over $A$. Then
$$C^*(\psi,\pi)=\overline{span}\{\psi^n(x)\psi^m(y)^*:x\in X^{\otimes n},
y\in X^{\otimes m}\}$$
By Lemma 2.4 of \cite{Katsura1},
for each $n\in\mathbb N$ there is a homomorphism $\psi^{(n)}:
\mathcal K(X^{\otimes n})\to B$ such that:
\begin{enumerate}
\item
$\pi(a)\psi^{(n)}(k)=\psi^{(n)}(\phi(a)k)$ for all $a\in A$ and all $k\in
\mathcal K(X^{\otimes n})$.
\item
$\psi^{(n)}(k)\psi(x)=\psi(kx)$ for all $x\in X$ and all $k\in \mathcal 
K(X^{\otimes n})$
\end{enumerate}
We define
the \emph{Katsura ideal} of $A$ to be the ideal:
$$J_X=\{a\in A:\phi(a)\in \mathcal K(X) \text{ and }ab=0\text{ for all
}b\in ker(\phi)\}$$
Where $\phi$ is the left action. This is often written as
as $J_X=\phi^{-1}\big(\mathcal K(X)\big)\cap\big(ker(\phi)\big)^\perp$. 
In many cases of interest, the left action of a correspondence is injective
and implemented by compact operators.
In this case we have $J_X=A$.

A Toeplitz representation is said to be \emph{Cuntz-Pimsner covariant}
if $\psi^{(1)}\big(\phi(a)\big)=\pi(a)$ for all $a\in J_X$. The 
\emph{Cuntz-Pimsner algebra} $\mathcal O_X$ is the quotient of 
$\mathcal T_X$ by the ideal generated by 
$$\{i_X^{(1)}\big(\phi(a)\big)-i_A(a):a\in J_X\}$$
It can be shown that $\mathcal O_X$ is generated by a universal Cuntz-Pimsner
covariant representation $(k_X,k_A)$.

\subsection{Actions and Coactions of Discrete Groups}
Working with actions and coactions of locally compact groups and their 
crossed products can be somewhat complicated, for a general reference we
suggest the appendix of \cite{enchilada}. Restricting attention
to discrete groups affords many simplifications. In this section, we will
summarize some of these simplifications.

\begin{prop}\label{CrossedByActionDiscrete}
Let $G$ be a discrete group and let $\alpha$ be an action of $G$ on a 
$C^*$-algebra $A$. Let $(i_A,i_G)$ be the canonical representation of the
system $(A,G,\alpha)$ in the reduced crossed product $A\rtimes_{\alpha,r}
G$. Then $A\rtimes_{\alpha,r} G$ is the closed linear span of elements of
the form $i_A(a)i_G(s)$ where $a\in A$ and $s\in G$. These have the following
algebraic properties:
\begin{align}
\big(i_A(a)i_G(s)\big)\big(i_A(b)i_G(t)\big)&=i_A\big(a\alpha_s(b)\big)i_G(st)\\
\big(i_A(a)i_G(s)\big)^*&=i_A\big(\alpha_{s^{-1}}(a^*)\big)i_G(s^{-1})
\end{align}
\end{prop}
\begin{proof}
The fact that the $i_A(a)i_G(s)$ densely span $A\rtimes_{\alpha,r}G$ follows
from the definition of the crossed product. The algebraic properties follow
easily from the fact that $i_A(a)i_G(s)=i_G(s)i_A\big(\alpha_s(a)\big)$ which
we obtain from $i_G(s)^*i_A(a)i_G(s)=i_A\big(\alpha_s(a)\big)$.
\end{proof}
\begin{prop}\label{CorrCrossedByActionDiscrete}
Let $(\gamma,\alpha)$ be an action of a discrete group $G$ on a correspondence
$(X,A)$. Let $(i_X,i_A,i^X_G,i^A_G)$ be the canonical representation
of the system in the crossed product $(X\rtimes_{\gamma,r}G,A\rtimes_{\alpha,r}
G)$. Then $X\rtimes_{\gamma,r}G$ is the closed linear span of $i_X(x)i_G^X(s)$
where $x\in X$ and $s\in G$. These satisfy the following algebraic properties:
\begin{align}
\big(i_X(x)i^X_G(s)\big)\big(i_A(a)i^A_G(t)\big)&=i_X\big(x\alpha_s(a)\big)i^X_G(st)
\label{ActionRmod}\\
\big\langle i_X(x)i_G^X(s),i_X(y)i^Y_G(t)\big\rangle_{A\rtimes_{\alpha,r}G}&=i_A\big(
\alpha_{s^{-1}}\big(\langle x,y\rangle_A\big)\big)i^A_G(s^{-1}t)
\label{ActionIP}\\
\big(i_A(a)i_G^A(s)\big)\big(i_X(x)i_G^X(t)\big)&=i_X\big(a\gamma_s(x)\big)i_G^X(st)
\label{ActionLmod}
\end{align}
\end{prop}
\begin{proof}
The fact that the $i_X(x)i^X_G(s)$ densely span $X\rtimes_{\gamma,r}G$
follows from the definition of the crossed product. To understand the algebraic
properties, recall from Lemma 3.3 of \cite{enchilada} that $L(X\rtimes_{\gamma,r}
G)\cong L(X)\rtimes_{\nu,r}G$ where $\nu$ is the coaction on $L(X)$ induced
by $(\gamma,\alpha)$. (\ref{ActionRmod}) and (\ref{ActionIP}) 
are then easily deduced
by applying the previous proposition to $L(X)\rtimes_{\nu,r}G$. (\ref{ActionLmod})
follows from the fact that the left action must be covariant with respect
the action.
\end{proof}
\begin{cor}\label{GeneratingSystemCrossed}
The sets 
\begin{align*}
(X\rtimes_{\gamma,r}G)_0&:=\{i_X(x)i_G^X(s):x\in X, s\in G\}\\
(A\rtimes_{\alpha,r}G)_0&:=\{i_A(a)i_G^A(s):a\in A,s\in G\}
\end{align*}
form a generating system for $X\rtimes_{\gamma,r}G$ in the sense of Definition
\ref{GeneratingSystem}.
\end{cor}
The simplification of crossed products by coactions of discrete groups come
from the realization that coactions by discrete groups can be viewed as 
gradings. This idea is presented in detail in \cite{Quigg}, but we briefly
summarize the main points in the next two propositions. We refer the reader
to \cite{Quigg} for the proofs.
\begin{prop}\label{CoactionGradingAlgebra}
Let $\delta:A\to M\big(A\otimes C^*(G)\big)$ be a coaction of a discrete
group $G$ on a $C^*$-algebra $A$. Then $A=\overline{\emph{span}}\{A_s\}_{
s\in G}$ where $A_s=\{a\in A:\delta(a)=a\otimes u_s\}$. Furthermore,
\begin{align}
A_s\cdot A_t&\subseteq A_{st}\\
A_s^*&=A_{s^{-1}}
\end{align}
\end{prop}
\begin{prop}\label{AlgebraCrossedByCoaction}
Let $\delta$ be a coaction of a discrete group $G$ on a $C^*$-algebra
$A$ and let $(j_A,j_G)$ be the canonical representations of $A$ and
$c_0(G)$ in the crossed product $A\rtimes_\delta G$. Then $A\rtimes_
\delta G$ is densely spanned by elements of the form $j_A(a_s)j_G(f)$
where $a_s\in A_s$, $f\in c_0(G)$ and 
These satisfy the following relations:
\begin{align}
\big(j_A(a)j_G(f)\big)\big(j_A(a_{s})j_G(g)\big)
&=j_A(aa_s)j_G\big(\lambda_{s^{-1}}(f)g\big)\\
\big(j_A(a_s)j_G(f)\big)^*&=j_A(a_s^*)j_G\big(\lambda_s(\overline{f})\big)
\end{align}
where $\lambda_s$ denotes left translation by $s$ on $c_0(G)$.
\end{prop}
Applying these propositions to linking algebras helps us to understand
coactions on correspondences:
\begin{prop}\label{CorrGrading}
Let $(\sigma,\delta)$ be a coaction of a discrete group $G$ on
a Hilbert module $(X,A)$. Then $X=\overline{\emph{span}}\{X_s\}_{s\in G}$
where $X_s=\{x\in X:\sigma(x)=x\otimes u_s\}$. Further:
\begin{align}
X_s\cdot A_t&\subseteq X_{st}\label{gradingRmod}\\
\langle x_s,x_t\rangle_A&\in A_{s^{-1}t}\label{gradingIP}\\
A_s\cdot X_t&\subseteq X_{st}\label{gradingLmod}
\end{align}
\end{prop}
\begin{proof}
Let $\varepsilon$ be the induced coaction on $L(X)$. Then we have a grading
$L(X)=\overline{\text{span}}\{L(X)_s\}_{s\in G}$. Since $p$ and $q$ are 
in fixed points of the coaction, if $z\in L(X)_s$ then $qzq\in L(X)_s$ 
and $pzq\in L(X)_s$. Recall that the restriction of $\varepsilon$ to 
$qL(X)q=
\begin{bmatrix}
0 & 0\\
0 & A
\end{bmatrix}
\cong A$ is $\delta$. Thus if $a$ is the element of $A$ corresponding
to $qzq$ then $\varepsilon(qzq)=(qzq)\otimes u_s$ if and only if
$\delta(a)=a\otimes u_s$. Thus $qL(X)_sq=
\begin{bmatrix}
0 & 0\\
0 & A_s
\end{bmatrix}\cong A_s$. Similar reasoning shows that $pL(X)_sq=
\begin{bmatrix}
0 & X_s\\
0 & 0
\end{bmatrix}\cong X_s$. (\ref{gradingRmod}) and (\ref{gradingIP}) 
then follow from multiplication in $L(X)$ together with the grading of $L(X)$.
(\ref{gradingLmod}) follows from the fact that the left module action is
covariant with respect to the coaction.
\end{proof}
Lemma 3.4 of \cite{enchilada} shows that $L(X\rtimes_\sigma G)\cong L(X)\rtimes_
\varepsilon G$. This fact, together with the preceding propositions, gives
us the following characterization of $X\rtimes_\sigma G$ in the case where
$G$ is discrete:
\begin{prop}\label{CorrCrossedByCoaction}
Let $(\sigma,\delta)$ be a coaction of a discrete group $G$ on a correspondence
$(X,A)$. Let $(j_X,j_A,j_G^X,j_G^A)$ be the canonical representation
of the system in the correspondence $(X\rtimes_\sigma G,A\rtimes_\delta G)$.
Then $X\rtimes_\sigma G$ is densely spanned by elements of the form 
$j_X(x_s)j_G^X(f)$ where $x_s\in X_s$ and $f\in c_0(G)$.
These elements satisfy the following relations:
\begin{align}
\big(j_X(x)j_G^X(f)\big)\big(j_A(a_s)j_G^A(g)\big)
&=j_X(xa_s)j_G^X(\lambda_{s^{-1}}(f)g)\\
\big\langle j_X(x_s)j_G^X(f),j_X(x_t)j_G^X(g)\big\rangle
_{A\rtimes_\delta G}&=j_A\big(\langle
x_s,x_t\rangle_A\big)j_G^A\big(\lambda_{t^{-1}s}(f)g\big)\label{CoactionIP}\\
\big(j_A(a)j_G^A(f)\big)\big(j_X(x_s)j_G^X(g)\big)&=j_X(ax_s)j_G^X\big(\lambda
_{s^{-1}}(f)g\big)\label{CoactionLmod}
\end{align}
\end{prop}
\begin{cor}\label{GeneratingSetCoCrossed}
The sets 
\begin{align*}
(X\rtimes_\sigma G)_0&:=\{j_X(x_s)j_G^X(f):x_s\in X_s,f\in c_0(G)\}\\
(A\rtimes_\delta G)_0&:=\{j_A(a_s)j_G^A(f):a_s\in A_s,g\in c_0(G)\}
\end{align*}
form a generating system for $X\rtimes_\sigma G$.
\end{cor}
\subsection{Quantum Groups}
There are many different notions of a quantum group. Our quantum groups
will be quantum groups generated by modular multiplicative unitaries.
We will briefly recall some of the basic facts about
these quantum groups and refer the reader to \cite{Timmermann}
for a more in depth overview of the subject.

\begin{dfn}
Given a separable Hilbert space $\mathcal H$, a \emph{multiplicative unitary}
$\mathbb W$ is a unitary operator on $\mathcal H\otimes\mathcal H$ such that
\begin{align}
\mathbb W_{23}\mathbb W_{12}=\mathbb W_{12}\mathbb W_{13}\mathbb W_{23}
\in\mathcal U(\mathcal H\otimes\mathcal H\otimes\mathcal H)
\label{Pentagon} 
\end{align}
where $\mathbb W_{ij}$ indicates that we are applying $\mathbb W$ to
the $i^{th}$ and $j^{th}$ factors of $\mathcal H$ and leaving the other 
fixed. Equation \ref{Pentagon} is sometimes referred to as the 
\emph{pentagon equation}.

$\mathbb W$ is called \emph{modular} if there exist (possibly unbounded)
operators $\widehat Q$ and $Q$ on $\mathcal H$ and a unitary $\widetilde
{\mathbb W}\in\mathcal U(\overline{\mathcal H}\otimes\mathcal H)$
(where $\overline{\mathcal H}$ is the dual space of $\mathcal H$) such that
\begin{enumerate}
\item
$\widehat Q$ and $Q$ are positive and self-adjoint with trivial kernels
\item
$\mathbb W^*(\widehat Q\otimes Q)\mathbb W=\widehat Q\otimes Q$
\item
$\langle \eta'\otimes\xi',\mathbb W(\eta\otimes\xi)\rangle=
\langle\overline\eta\otimes Q\xi',\widetilde{\mathbb W}(\overline
{\eta'}\otimes Q^{-1}\xi)\rangle$ for all $\xi\in\text{Dom}(Q^{-1})$,
$\xi'\in\text{Dom}(Q)$, and $\eta,\eta'\in\mathcal H$
\end{enumerate}
\end{dfn}

\begin{ex}[Example 7.1.4 of \cite{Timmermann}]\label{GroupMU}
Let $G$ be a locally compact group. We can identify $L^2(G)\otimes
L^2(G)$ with $L^2(G\times G)$ and define
$\mathbb W_G\in\mathcal B\big(L^2(G)\otimes L^2(G)\big)$ by
$$(\mathbb W_G\zeta)(s,t):=\zeta(s,s^{-1}t)$$
Then $\mathbb W_G$ is a multiplicative unitary.
\end{ex}
\begin{thm}[Theorem 2.7 of \cite{Woronowicz}]\label{QGconstruction}
For a modular multiplicative unitary $\mathbb W\in\mathcal U(
\mathcal H\otimes\mathcal H)$, set
\begin{align*}
S&:=\overline{\emph{span}}\{(\omega\otimes\emph{id}_\mathcal H)\mathbb W:\omega
\in\mathcal B(\mathcal H)_*\}\\
\widehat S&:=\overline{\emph{span}}\{(\emph{id}_\mathcal H\otimes\omega)
\mathbb W:\omega\in\mathcal B(\mathcal H)_*\}
\end{align*}
Then:
\begin{itemize}
\item $S$ and $\widehat S$ are separable, nondegenerate $C^*$-subalgebras
of $\mathcal B(\mathcal H)$.
\item $\mathbb W\in\mathcal U(\widehat S\otimes S)\subseteq\mathcal U(\mathcal
H\otimes\mathcal H)$. When we wish to view $\mathbb W$ as a unitary multiplier
of $\widehat S\otimes S$ we will denote it by $W^S$ and refer to it as the
\emph{reduced bicharacter} of $S$.
\item
There are unique homomorphisms $\Delta_S,\widehat\Delta_S:S\to S\otimes S$
such that
\begin{align*}
(\emph{id}_{\widehat S} \otimes \Delta_S)W^S&=W^S_{12}W^S_{13}\in\mathcal
U(\widehat
S\otimes S\otimes S)\\
(\widehat\Delta_S\otimes\emph{id}_S)W^S&=W^S_{23}W^S_{13}\in\mathcal U(\widehat
S\otimes \widehat S\otimes S)
\end{align*}
\item $(S,\Delta_S)$ and $(\widehat S,\widehat\Delta_S)$ are $C^*$-bialgebras
\end{itemize}
\end{thm}
\begin{dfn}
\label{QGdefinition}A \emph{quantum group} is a $C^*$-bialgebra $\mathbb
G:=(S,\Delta_S)$
arising from a modular multiplicative unitary as in Theorem
\ref{QGconstruction}. The \emph{dual} 
$\widehat{\mathbb G}$ of $\mathbb G$ is the bialgebra
$(\widehat S,\widehat\Delta_S)$. 
$\widehat{\mathbb G}$ is generated by the modular multiplicative 
unitary $\widehat{\mathbb W}:=\Sigma\mathbb W^*\Sigma$ (where
$\Sigma$ is the flip isomorphism: $x\otimes y\mapsto y\otimes x$).
$\widehat{\mathbb W}$ is called the \emph{dual} of $\mathbb W$.
We will sometimes write $S(\mathbb W)$ and $\widehat S(\mathbb W)$
for $S$ and $\widehat S$ if we wish to call attention to the multiplicative
unitary which generated $S$ or $\widehat S$.
\end{dfn}
\begin{ex}[Example 7.2.13 of \cite{Timmermann}]\label{QGfromGroup}
Let $G$ be a locally compact group and let $\mathbb W_G$
be the multiplicative unitary described in Example \ref{GroupMU}. Then
$S(\mathbb W_G)\cong C_r^*(G)$ and $\widehat S(\mathbb W_G)\cong C_0(G)$
as $C^*$-bialgebras. We denote the reduced bicharacter associated to $\mathbb
W_G$ by $W_G$.
\end{ex}
\begin{dfn}
Given two quantum groups $\mathbb G=(S,\Delta_S)$ and $\mathbb H
=(T,\Delta_T)$, we define a \emph{bicharacter from $\mathbb G$ to
$\widehat{\mathbb H}$} to be a unitary $\chi\in\mathcal U(\widehat
S\otimes\widehat T)$ such that 
\begin{align}
(\widehat\Delta_S\otimes\text{id}_{\widehat T})\chi&=\chi_{23}\chi_{13}\in\mathcal
U(\widehat S\otimes\widehat S\otimes \widehat T)\\
(\text{id}_{\widehat S}\otimes\widehat\Delta_T)\chi&=\chi_{13}\chi_{13}\in\mathcal
U(\widehat S\otimes\widehat S\otimes \widehat T)
\end{align} 
\end{dfn}
Bicharacters are used in the construction of the twisted tensor product.
The reduced bicharacter from Theorem \ref{QGconstruction} is a
special case of a bicharacter. In fact the only bicharacter we will
need for our simplified twisted tensor products is the bicharacter
$W_G$ associated to the multiplicative unitary $\mathbb W_G$ from
Example \ref{GroupMU}.
\begin{dfn}
A \emph{continuous coaction} of a quantum group $\mathbb G=(S,\Delta_S)$
on a $C^*$-algebra $A$ is a homomorphism $\delta:A\to M(A\otimes S)$ such
that
\begin{enumerate}
\item
$\delta$ is 1-1
\item
$(\text{id}_A\otimes\Delta_S)\delta=(\delta\otimes\text{id}_S)\delta$
\item
$\delta(A)\cdot(1_A\otimes S)=A\otimes S$
\end{enumerate}
We will sometimes refer to $A$ as a $\mathbb G$-$C^*$-algebra.
\end{dfn}
We can also define coactions of quantum groups on $C^*$-correspondences:
\begin{dfn}
A \emph{$\mathbb G$-equivariant $C^*$-correspondence} over a $\mathbb G$
-$C^*$-algebra $(A,\delta)$ is a $C^*$-correspondence $X$ over $A$ with a
coaction $\sigma:X\to M(X\otimes S)$ such that
\begin{enumerate}
\item
$\sigma(x)\delta(a)=\sigma(xa)$ and 
$\delta(a)\sigma(x)=\sigma(ax)$ for $x\in X$ and $a\in A$
\item
$\delta\big(\langle x,y\rangle_A\big)=\langle\sigma(x),\sigma(y)\rangle
_ {M(A\otimes S)}$
\item
$\sigma(x)\cdot(1\otimes S)= X\otimes S$
\item
$(1\otimes S)\cdot\sigma(x)=X\otimes S$
\item
$(\sigma\otimes\text{id}_A)\sigma=(\text{id}_X\otimes\sigma)\sigma$
\end{enumerate}

\end{dfn}

\subsection{Twisted Tensor Products}
Throughout this section, $\mathbb G=(S,\Delta_S)$ and $\mathbb H=(T,\Delta_T)$
will be quantum groups, $\chi$ will be a bicharacter from
$\mathbb G\to\widehat{\mathbb H}$, and $A$ and $B$ will be $C^*$-algebras
carrying coactions $\delta_A$ and $\delta_B$ of $\mathbb G$ and $\mathbb H$ respectively.
\begin{dfn}[Definition 3.1 of \cite{Woronowicz}]\label{HPdefinition}
A \emph{$\chi$-Heisenberg pair} (or simply \emph{Heisenberg pair})
is a pair of representations $\pi:S\to\mathcal B(\mathcal H)$
and $\rho:T\to\mathcal B(\mathcal H)$ such that
$$W_{1\pi}^SW_{2\rho}^T=W_{2\rho}^TW_{1\pi}^S\chi_{12}\in
\mathcal U\big(\widehat S\otimes\widehat T\otimes\mathcal K(\mathcal H)\big)
$$ where $W^S_{1\pi}=\big((\text{id}_{\widehat S}\otimes\pi)W^S\big)_{13}$
and $W^T_{2\rho}=\big((\text{id}_{\widehat T}\otimes\rho)W^T\big)_{23}$.
\end{dfn}
\begin{dfn}\label{TPdef}
Suppose $(\pi,\rho)$ is a Heisenberg pair. Define maps
\begin{align*}
i_A:A&\to M\big(A\otimes B\otimes\mathcal K(\mathcal H)\big)\\
i_B:B&\to M\big(A\otimes B\otimes\mathcal K(\mathcal H)\big)
\end{align*}
as follows:
\begin{align*}
i_A(a)&=(\text{id}_A\otimes\pi)\delta_A(a)_{13}\\
i_B(b)&=(\text{id}_B\otimes\rho)\delta_B(b)_{23}
\end{align*}
It is shown in Lemma 3.20 of \cite{Woronowicz} that 
$A\boxtimes_\chi B:=i_A(A)\cdot i_B(B)$
is a $C^*$-subalgebra of $A\otimes B\otimes\mathbb K(\mathcal H)$ and that,
up to isomorphism, $A\boxtimes_\chi B$ does not depend upon the choice of
Heisenberg pair. We refer to $A\boxtimes_\chi B$ as the \emph{twisted
tensor product} of $A$ and $B$ and we write $a\boxtimes b$ for $i_A(a)i_B(b)$.
\end{dfn}
This construction can also be extended to correspondences. Recall that, 
given a $C^*$-correspondence $X$ over a $C^*$-algebra $A$, the 
\emph{linking algebra} $L(X)$ is the algebra $\mathbb K(X\oplus A)$. 
This is
often thought of in terms of it's block matrix form:
$\begin{bmatrix}\mathcal K(X) & X \\
\overline X & A \\
\end{bmatrix}$. We will make use of the following proposition:
\begin{prop}[Proposition 2.7 of \cite{Baaj}]
Let $\delta_{L(X)}:L(X)\to M(L(X)\otimes S)$ be a coaction of $\mathbb G$
on $L(X)$ such that the inclusion $j_A:A\to L(X)$ is $\mathbb G$-equivariant.
Then there exists a unique coaction $\sigma$ on
$X$ such that $j_{M(X\otimes S)}\circ\sigma=\delta_{L(X)}\circ j_X$ where
$j_{M(X\otimes S)}$ is the inclusion $M(X\otimes S)\to M(L(X)\otimes S)$
and $j_X$ is the inclusion $X\to L(X)$.

Conversely, if $\sigma$ is a coaction of $\mathbb G$ on $X$,
then there is a unique coaction $\delta_{L(X)}$ of $\mathbb G$ on $L(X)$
such
that $j_{M(X\otimes S)}\circ\sigma=\delta_{L(X)}\circ j_X$ and such that 
$j_A:A\to L(X)$
is $\mathbb G$-equivariant.
\end{prop}
We refer the reader to \cite{Baaj} for the proof.
\begin{dfn}\label{TwistedCorr}
Let $(X,A)$ and $(Y,B)$ be $C^*$-correspondences with coactions
of $\mathbb G$ and $\mathbb H$ respectively and let $\delta_{L(X)}$
and $\delta_{L(Y)}$ be the induced coactions of $L(X)$ and $L(Y)$.
We can form the twisted tensor product $L(X)\boxtimes_\chi L(Y)$.
Viewing $X$ and $Y$ as subspaces of $L(X)$ and $L(Y)$,
we define
$$X\boxtimes_\chi Y:=\iota_{L(X)}(X)\cdot\iota_{L(Y)}(Y)$$
Proposition 5.10 of \cite{Woronowicz}, and the discussion which follows it,
shows that this is a
correspondence over $A\boxtimes_\chi B$ with left action given by
$\phi_X\boxtimes\phi_Y$. It also shows that 
\begin{align}
\mathcal K(X\boxtimes_\chi
Y)\cong \mathcal K(X)\boxtimes_\chi \mathcal K(Y)\label{TwistedProdOfCompacts}
\end{align}
a fact which we shall make use of later.
\end{dfn}

\section{Discrete Group Twisted Tensor Products}
\subsection{Basics}
In what follows we will restrict our attention to the following
special case of the twisted tensor product construction:
\begin{dfn}
Suppose that $G$ is a discrete group, $(A,G,\alpha)$ is a $C^*$-dynamical
system, and $(B,G,\delta)$ is a coaction. Let $\delta^\alpha:A\to {M(A\otimes
c_0(G)\big)}$ be the coaction of $c_0(G)$ (as a quantum group) on $A$ associated
to $\alpha$ as in Theorem 9.2.4 of \cite{Timmermann}. We can view $\delta$
as a coaction of the quantum group $C_r^*(G)$ and we can form the twisted tensor
product $A\boxtimes_{W_G} B$ where $W_G$ is the multiplicative unitary of
Example \ref{GroupMU} viewed as a bicharacter from $c_0(G)$ to $C_r^*(G)$
(in other words the reduced bicharacter of the quantum group $C_r^*(G)$).
We refer to this special case as a \emph{discrete group twisted tensor product}.
Since this construction depends only upon the action and coaction, we
will sometimes write $A\tensor[_\alpha]\boxtimes{_\delta} B$ for $A\boxtimes_{W_G}
B$. We can also define a $C^*$-algebra by $B\tensor[_\delta]\boxtimes{_\alpha}
A=B\boxtimes_{\widehat W_G}A$. It is easy to see that the map $a\boxtimes
b\mapsto b\boxtimes a$ extends to an isomorphism $A\tensor[_\alpha]\boxtimes
{_\delta} B\cong B\tensor[_\delta]\boxtimes{_\alpha}A$.
\end{dfn}
The main reason that this special case is of interest is that we can write
down
a precise formula for the multiplication and involution of certain elementary
tensors. To understand this, we must recall that the coaction $\delta$ of
a discrete group $G$ on $B$ gives rise to a $G$-grading of $B$. That is,
there exist subspaces $\{B_s\}_{s\in G}$ such that
\begin{enumerate}
\item
$\overline{\text{span}}(B_s)=B$
\item
$B_s\cdot B_t\subseteq B_{st}$
\item
$B_s^*=B_{s^{-1}}$.
\end{enumerate}
Specifically, $B_s=\{b\in B:\delta(b)=b\otimes u_s\}$. With this in mind,
we present the following:
\begin{prop}\label{Formulas}
Given a $C^*$-dynamical system $(A,G,\alpha)$ and a coaction $(B,G,\delta)$,
let $a,a'\in A$, $b_s\in B_s$ and $b\in B$. Then, in the twisted tensor
product $A\tensor[_\alpha]\boxtimes{_\delta}B$ we have:
\begin{align*}
(a\boxtimes b_s)(a'\boxtimes b)&=a\alpha_s(a')\boxtimes b_sb\\
(a\boxtimes b_s)^*&=\alpha_{s^{-1}}(a)^*\boxtimes b_s^*
\end{align*}
\end{prop}
Before we prove this, we will need the following:
\begin{lem}
Let $G$ be a locally compact group and let $m:C_0(G)\to\mathbb B\big(L^2(G)\big)$
be the left action of $C_0(G)$ on $L^2(G)$ by multiplication of functions
in $L^2(G)$. Let $\lambda:C^*(G)\to\mathbb B\big(L^2(G)\big)$ be the left
regular representation. Then $(m,\lambda)$ is a $W_G$-Heisenberg pair where
$W_G$ is the reduced bicharacter of $G$ as in Example \ref{QGfromGroup}
\end{lem}
\begin{proof}
This is a special case of Example 3.9 of \cite{Woronowicz}.
\end{proof}
\begin{lem}\label{Commutation}
In the situation of the above proposition, let $$i_A:A\to
M\big(A\otimes B\otimes\mathbb K(L^2(G))\big)$$
$$i_B:B\to M\big(A\otimes B\otimes\mathbb K(L^2(G))\big)$$
be the maps associated to the Heisenberg pair $(m,\lambda)$ as described
in Definition \ref{TPdef}.
Then for any $a\in A$, $s\in G$ and $b_s\in B_s$ we have
$$i_B(b_s)i_A(a)=i_A\big(\alpha_{s^{-1}}(a)\big)i_B(b_s)$$
\end{lem}
\begin{proof}
Recall that $(m,\lambda)$ is a covariant homomorphism 
$(c_0(G),G,\sigma)\to\mathbb B(L^2(G))$ 
where sigma is left translation
in $c_0(G)$. Also, $(\delta^\alpha,\sigma_2)$ is a covariant homomorphism
$(A,G,\alpha)\to M\big(A\otimes c_0(G)\big)$. Thus ${\big((\text{id}_A
\otimes m)\circ
\delta^\alpha,\lambda_2\big)}$ is a covariant homomorphism $(A,G,\alpha)
\to M\big(A\otimes \mathbb B(L^2(G))\big)$. Thus 
$$\lambda_2(s)^*\Big((\text{id}_A\otimes m)\circ\delta^\alpha(a)\Big)
\lambda_2(s)=(\text{id}_A\otimes m)\circ\delta^\alpha\big(\alpha_s(a)\big)$$
or equivalently
$$\Big((\text{id}_A\otimes m)\circ\delta^\alpha\big(\alpha_{s^{-1}}(a)\big)
\Big)\lambda_2(s)=\lambda_2(s)
\Big((\text{id}_A\otimes m)\circ\delta^\alpha(a)\big)\Big)$$
With this in mind, we notice the following:
\begin{align*}
i_B(b_s)i_A(a)&=(\text{id}_B\otimes\lambda)\delta(b_s)_{23}
(\text{id}_A\otimes m)\delta^\alpha(a)_{13}\\
&=\big(1_A\otimes b_s\otimes\lambda(s)\big)
(\text{id}_A\otimes m)\delta^\alpha(a)_{13}\\
&=(b_s)_2\big(\lambda(s)\big)_3(\text{id}_A\otimes m)\delta^\alpha
(a)_{13}\\
&=(b_s)_2(\text{id}_A\otimes m)\delta^\alpha\big(\alpha_{s^{-1}}(a)\big)_{13}
\big(\lambda(s)\big)_3\\
&=(\text{id}_A\otimes m)\delta^\alpha\big(\alpha_{s^{-1}}(a)\big)_{13}
(b_s)_2\big(\lambda(s)\big)_3\\
&=i_A\big(\alpha_{s^{-1}}(a)\big)i_B(b_s)
\end{align*}
\end{proof}
We can now prove Proposition \ref{Formulas}.
\begin{proof}\emph{(of Proposition \ref{Formulas})}

We have:
\begin{align*}
(a\boxtimes b_s)(a'\boxtimes b)&=i_A(a)i_B(b_s)i_A(a')i_B(b)\\
&=i_A(a)i_A\big(\alpha_s(a')\big)i_B(b_s)i_B(b)\\
&=a\alpha_s(a')\boxtimes b_sb
\end{align*}
and
\begin{align*}
(a\boxtimes b_s)^*&=\big(i_A(a)i_B(b_s)\big)^*\\
&=i_B(b_s^*)i_A(a^*)\\
&=i_A\big(\alpha_{s^{-1}}(a)^*\big)i_B(b_s^*)\\
&=\alpha_{s^{-1}}(a)^*\boxtimes b_s^*
\end{align*}
\end{proof}
We also have simple formulas for the algebraic properties of
twisted tensor products of correspondences:
\begin{prop}
Let $(X,A)$ be a correspondence with an action $(\gamma,\alpha)$ of
$G$ and $(Y,B)$ be a correspondence with a coaction $(\sigma,\delta)$
of $G$.
Let $\alpha_{L(X)}$ be the action of $G$ on $L(X)$ induced by the
action of $G$ on $X$ and let $\delta_{L(Y)}$
be the coaction of $G$ on $L(Y)$ induced by $(\sigma,\delta)$.
We can form the correspondence 
$X\tensor[_\gamma]\boxtimes{_\sigma}Y:=X\boxtimes_{W_G}Y
\subseteq L(X)\boxtimes_{W_G}L(Y)=L(X)\tensor[_{\alpha_{L(X)}}]\boxtimes{_
{\delta_{L(Y)}}}L(Y)$
as in Definition
\ref{TwistedCorr}. Then
\begin{enumerate}
\item
$(a\boxtimes b_s)(x\boxtimes y)=a\gamma_s(x)\boxtimes b_s y$
\item
$(x\boxtimes y_s)(a\boxtimes b)=x\alpha_s(a)\boxtimes y_sb$
\item
$\langle x\boxtimes y_s,x'\boxtimes y\rangle_{A\boxtimes B}
=\alpha_{s^{-1}}\big(\langle x,x'\rangle_A\big)\boxtimes\langle y_s,y\rangle_B$
\end{enumerate}
\end{prop}
\begin{proof}
All of these facts follow from translating to multiplication
in $L(X)\boxtimes L(Y)$ and applying Lemma \ref{Commutation}. 
\end{proof}
\begin{cor}\label{TwistedGeneratingSystem}
Let $(X,A)$ be a correspondence with an action $(\gamma,\alpha)$ of
$G$ and $(Y,B)$ be a correspondence with a coaction $(\sigma,\delta)$
of $G$. Suppose $(X^0,A^0)$ is a generating system for $(X,A)$ which is 
stable with respect to the group action. That is, $\gamma_s(X^0)\subseteq
X^0$ and $\alpha_s(A^0)\subseteq A^0$ for all $s\in G$. Suppose further that
$(Y^0,B^0)$ is a generating system for $(Y,B)$ such that the elements of
$Y^0$ and $B^0$ are homogenous with respect to the grading. That is, for
all $y\in Y^0$ there is $s\in G$ such that $y\in Y_s$ and for all $b\in B^0$
there is $t\in G$ such that $b\in B_t$. Then
\begin{align*}
(X\boxtimes Y)_0&:=\{x\boxtimes y:x\in X^0, y\in Y^0\}\\
(A\boxtimes B)_0&:=\{a\boxtimes b:a\in A^0, b\in B^0\}
\end{align*}
form a generating system for $X\boxtimes Y$.
\end{cor}
\begin{proof}
It is clear from the bilinearity of the twisted tensor product that
$\overline{\text{span}}\big((X\boxtimes Y)_0\big)=X\boxtimes Y$ and
$\overline{\text{span}}\big((A\boxtimes B)_0\big)=A\boxtimes B$. To see
that $(X\boxtimes Y)_0$ is stable under the left and right actions of 
elements of $(A\boxtimes B)_0$, let $x\boxtimes y\in (X\boxtimes Y)_0$
and let $a\boxtimes b\in (A\boxtimes B)_0$. By definition, we must
have that $b\in B_s$ and $y\in Y_t$ for some $s,t\in G$. Thus we have:
\begin{align*}
(a\boxtimes b)(x\boxtimes y)&=a\gamma_s(x)\boxtimes by\\
(x\boxtimes y)(a\boxtimes b)&=x\alpha_t(a)\boxtimes yb
\end{align*}
Since $X^0$ and $A^0$ are stable under the action of $G$, $\gamma_s(x)\in
X^0$ and $\alpha_t(a)\in A^0$. Thus $a\gamma_s(x),\gamma_t(a)x\in X^0$ so
$(X\boxtimes Y)_0$ is indeed stable under the left and right actions of
$(A\boxtimes B)_0$.
\end{proof}
\subsection{Examples}
In \cite{Woronowicz}, the authors show that if $A$ and $B$ are 
$\mathbb Z_2$-graded algebras, then the graded tensor product $A\widehat\otimes
B$ is isomorphic $A\boxtimes_{W^{\mathbb Z_2}}B$ (where the coactions on
$A$ and $B$ are the ones canonically associated with the grading) with the
map $a\widehat\otimes b\mapsto a\boxtimes b$ extending to an isomorphism.
Since $\mathbb Z_2$ is self-dual, the coaction of $\mathbb Z_2$ on $A$ gives
rise to an action $\alpha$ of $\mathbb Z_2$ on $A$. If we let $\delta$ denote
the coaction of $\mathbb Z_2$ on $\mathbb Z_2$ we see that $A\boxtimes_
{W^{\mathbb Z_2}}B=A\tensor[_\alpha]\boxtimes{_\delta}B$ so $A\widehat\otimes
B$ fits into our discrete group twisted tensor framework. The following
example shows that the graded external tensor product of graded correspondences
also fits into our framework.
\begin{ex}\label{GradedPorductCorr}
Let $A$ and $B$ be $\mathbb Z_2$-graded $C^*$ algebras and let $X$
and $Y$ be $\mathbb Z_2$-graded correspondences over $A$ and $B$
(i.e. $X$ and $Y$ are graded as Hilbert $A$- and $B$-modules respectively
and the left action maps $\phi_X$ and $\phi_Y$ are graded with respect
to the induced gradings on $\mathcal L(X)$ and $\mathcal L(Y)$).
Let $X^0=X_0\cup X_1$, $A^0=A_0\cup A_1$, $Y^0=Y_0\cup Y_1$ and $B^0
=B_0\cup B_1$. Then by Corollary \ref{GeneratingSetCoCrossed} $(X^0,A^0)$
and $(Y^0,B^0)$ are generating sets for $X$ and $Y$ respectively.
Consider the graded external tensor product $X\widehat\otimes Y$.
This is the closure of the algebraic tensor product $X\odot Y$
with respect to the norm associated to the inner product whose value
on generators is given by:
$$\langle x_1\widehat\otimes y_1,x_2\widehat\otimes y_2\rangle=(-1)
^{\partial y_1(\partial x_1+\partial x_2)}\langle x_1,x_2\rangle\widehat\otimes
\langle y_1,y_2\rangle$$
where $x_1,x_2\in X_0$ and $y_1,y_2\in Y_0$.
The left and right actions are given by:
\begin{align*}
(a\widehat\otimes b)(x\widehat\otimes y)&=(-1)^{\partial b\partial x}(ax
\widehat\otimes by)\\
(x\widehat\otimes y)(a\widehat\otimes b)&=(-1)^{\partial y\partial a}(xa
\widehat\otimes yb)
\end{align*} 
for $a\in A_0$, $b\in B_0$, $x\in X_0$, and $y\in Y_0$.
Thus
\begin{align*}
(A\widehat\otimes B)_0&:=\{a\widehat\otimes b:a\in A_0,b\in B_0\}\\
(X\widehat\otimes Y)_0&:=\{x\widehat\otimes y:x\in X_0,y\in Y_0\}, 
\end{align*}
is a generating system for $X\widehat\otimes Y$.
Now, let $(\gamma,\alpha)$ be the action of $\mathbb Z_2$ on $(X,A)$ 
associated to the grading of $X$ and let $(\sigma,\delta)$ be the coaction
of $\mathbb Z_2$ on $(Y,B)$ associated to the grading of $Y$. Consider the
associated twisted tensor product $X\boxtimes Y$. Note that the sets
$X_0$ and $A_0$ are stable under the actions $\gamma$ and $\alpha$ and that
$Y_0$ and $B_0$ consist of elements which are homogeneous with respect to
the gradings associated to $\sigma$ and $\delta$. Thus, by Corollary
\ref{TwistedGeneratingSystem} the sets
\begin{align*}
(X\boxtimes Y)_0&:=\{x\boxtimes y:x\in X_0,y\in Y_0\}\\
(A\boxtimes B)_0&:=\{a\boxtimes b:a\in A_0,b\in B_0\}
\end{align*}
form a generating system for $X\boxtimes Y$. Let
$\Phi_0:(X\widehat\otimes Y)_0\to (X\boxtimes Y)_0$
be the map $x\widehat\otimes y\mapsto x\boxtimes y$. This is clearly a
bijection. Let
$\varphi:A\widehat\otimes B\to A\tensor[_\alpha]\boxtimes{_\delta}B$
be the isomorphism described above. 
For $a\in A_0$, $b\in B_0$, $x\in X_0$, and $y\in Y_0$, we have:
\begin{align*}
\Phi_0\big((a\widehat\otimes b)(x\widehat\otimes y)\big)&=(-1)^{\partial
b\partial x}\Phi_0(ax\widehat\otimes by)\\
&=(-1)^{\partial b\partial x}(ax\boxtimes by)\\
&=a\gamma_{\partial b}(x)\boxtimes by\\
&=(a\boxtimes b)(x\boxtimes y)\\
&=\varphi(a\widehat\otimes b)\Phi_0(x\widehat\otimes y)
\end{align*}
and similarly
$$\Phi_0\big((x\widehat\otimes y)(a\widehat\otimes b)\big)=
\Phi_0(x\widehat\otimes y)\varphi(a\widehat\otimes b)$$
therefore $(\Phi_0,\varphi)$ preserves the left and right actions. Additionally,
for $x_1,x_2\in X_0$ and $y_1,y_2\in Y_0$, we have that:
\begin{align*}
\big\langle\Phi_0(x_1\widehat\otimes y_1),\Phi_0(x_2\widehat\otimes
y_2)\big\rangle&=\langle x_1\boxtimes y_1,x_2\boxtimes y_2\rangle\\
&=\alpha_{\partial y_1}\big(\langle x_1,x_2\rangle\big)\boxtimes\langle
y_1,y_2\rangle\big)\\
&=(-1)^{\partial y_1(\partial x_1+\partial x_2)}\big(\langle x_1,x_2\rangle
\boxtimes\langle y_1,y_2\rangle\big) \\
&=(-1)^{\partial y_1(\partial x_1+\partial x_2)}\varphi\big(\langle x_1,
x_2\rangle\widehat\otimes\langle y_1,y_2\rangle\big) \\
&=\varphi\big(\langle x_1\widehat\otimes x_2,y_1\widehat\otimes y_2\rangle
\big)
\end{align*}
therefore, by Lemma \ref{IsomLemma}, $\Phi_0$ extends to an isomorphism 
$\Phi:X\widehat\otimes Y\to X\boxtimes Y$. Thus $X\widehat\otimes Y\cong
X\boxtimes Y$.
\end{ex}
The following example can be viewed as a generalization of the 
skew graph construction presented in Chapter 6 of \cite{GraphAlgebraBook}.
\begin{ex}\label{GraphProductCorr}
Let $E=\{E^0,E^1,r_E,s_E\}$ and $F=\{F^0,F^1,r_F,s_F\}$ 
be directed graphs and let $G$ be a discrete group. Let $A:=c_0(E^0)$
and let $B:=c_0(F^0)$. 
Let $\alpha^E$ be an action of $G$ on $E$ by graph automorphisms and
let $\delta$ be a $G$-labeling of $F$, i.e. a map $\delta:F^1\to G$.
It is easy to see that the maps $f\mapsto f\circ\alpha^E_{s}$ on $A$
together with the maps $x\mapsto x\circ\alpha^G_s$ on $c_c(E^1)$ give rise
to group homomorphisms $G\to\text{Aut}(A)$ 
and $G\to\text{Aut}\big(c_c(E^1)\big)$ which in turn give rise to a group
action $(\gamma,\alpha)$ on $(X(E),A)$.
We also get a coaction $\sigma$ of $G$ on $X(F)$ from $\delta$. 
To see this,
recall that $c_c(F^1)$ is generated by the characteristic functions $\chi_e$
and define $\sigma(\chi_f):=\chi_f\otimes u_{\delta(f)}$. So $(\sigma,\iota)$
is a coaction on $\big(X(F),B\big)$ where $\iota$ is the 
trivial coaction on $B$.
We define a new directed graph
$E\tensor[_{\alpha^E}]\times{_\delta}F:=\{E^0\times F^0,E^1\times F^1,
r,s\}$ where $s(e\times f)=\alpha^E_{\delta(f)}\big(s_E(e)\big)\times s_F(f)$
and $r(e\times f)=r_E(e)\times r_F(f)$. We define $C:=c_0(E^0\times F^0)
\cong A\otimes B$.
We will show that $X\big(E\tensor
[_{\alpha^E}]\times{_\delta}F\big)\cong X(E)\tensor[_\gamma]\boxtimes{_\sigma}X(F)$.

Let $X(E)^0$ be the set of characteristic functions on $E^1$ and let
$A^0$ be the set of characteristic functions on $E^0$.
The characteristic functions densely span $c_c(E^1)$ which
is dense in $X(E)$. 
For $\chi_v\in A^0$ and $\chi_e\in X(E)^0$ we have that 
\begin{align*}
\chi_v\cdot\chi_e&=
\begin{cases}\chi_e & \text{if }r(e)=v \\
0 & \text{otherwise} \\
\end{cases}\in X(E)^0\\ 
\chi_e\cdot\chi_v&=
\begin{cases}
\chi_e & \text{if }s(e)=v\\
0 & \text{otherwise}
\end{cases}\in X(E)^0\\
\langle\chi_e,\chi_{e'}\rangle_A&=\begin{cases}
\chi_{s(e)} & \text{if }e=e'\\
0 & \text{otherwise}
\end{cases}
\end{align*}
therefore $\big(X(E)^0,A^0\big)$ 
is a generating system for $X(E)$. We define $\big(X(F)^0,B^0\big)$
and $\big(X\big(E\tensor[_{\alpha^E}]\times{_\delta}F\big)^0,C^0\big)$
is an analogous way, and we see that they are generating systems for 
$X(F)$ and $X(E\tensor[_{\alpha^E}]\times{_\delta}F)$ respectively.
Further, from the definition of the coaction $\sigma$ we have that
all characteristic functions on $F^1$ are homogeneous with respect to the
grading induced by $\sigma$ and since $\iota$ is trivial, the grading it
induces is also trivial so $A^0$ is trivially homogeneous. Also, the actions
$\alpha$ and $\gamma$ take generating functions to generating functions:
$\alpha_s:\chi_v\to\chi_{\alpha^E_{s^{-1}}(v)}$, $\gamma_s:\chi_e\to\chi_{\alpha^E_
  {s^{-1}}(e)}$ so the sets $X(E)^0$ and $A^0$ are fixed by the actions. This allows
us to apply Corollary \ref{TwistedGeneratingSystem} and deduce that the sets
\begin{align*}
\big(X(E)\boxtimes X(F)\big)^0&=\{\chi_e\boxtimes\chi_f:e\in E^1,f\in F^1\}\\
(A\otimes B)^0&=\{\chi_v\otimes\chi_w:v\in E^0,w\in F^0\}
\end{align*}
form a generating system for the twisted tensor product $X(E)\tensor
[_\gamma]\boxtimes{_\sigma}X(F)$. 

Let $\varphi:C\to A\otimes B$ and let
$\Phi_0:X\big(E\tensor[_{\alpha^E}]\times{_\delta}F\big)^0\to\big(X(E)\boxtimes
X(F)\big)^0$ be the map $\chi_{e\times f}\mapsto\chi_e\boxtimes\chi_f$.
Clearly $\Phi_0$ is bijective, we wish to show that it preserves the
inner product and left and right actions. Note that
\begin{align*}
\varphi\big(\big\langle\chi_{e\times f},\chi_{e'\times f'}\big\rangle_C
\big)&=\begin{cases}
\varphi\big(\chi_{s(e\times f)}\big) & \text{if }e\times f=e'\times f'\\
\phi(0) & \text{otherwise}
\end{cases}\\
&=\begin{cases}
\chi_{\alpha^E_{\delta(f)}(s_E(e))}\otimes\chi_{s_F(f)} & \text{if $e=e'$ and
$f=f'$}\\
0 & \text{otherwise}
\end{cases}\\
&=\begin{cases}
\alpha_{\delta(f)^{-1}}(\chi_{s_E(e)})\otimes\chi_{s_F(f)}&\text{if $e=e'$
and $f=f'$}\\
0 & \text{otherwise}
\end{cases}\\
&=\alpha_{\delta(f)^{-1}}\big(\langle\chi_e,\chi_{e'}\rangle_A\big)\boxtimes
\langle\chi_f,\chi_{f'}\rangle_B\\
&=\big\langle \chi_e\boxtimes\chi_f,\chi_{e'}\boxtimes\chi_{f'}\big\rangle
_{A\otimes B}\\
&=\big\langle\Phi_0(\chi_{e\times f}),\Phi_0(\chi_{e'\times f'})\big\rangle
_{A\otimes B}
\end{align*}
and
\begin{align*}
\Phi_0(\chi_{v\times w}\cdot\chi_{e\times f})&=\begin{cases}
\Phi_0(\chi_{e\times f})&\text{if }r(e\times f)=v\times w\\
\Phi_0(0) & \text{otherwise}
\end{cases}\\
&=\begin{cases}
\chi_e\boxtimes\chi_f & \text{if $r_E(e)=v$
and $r_F(f)=w$}\\
0&\text{otherwise}
\end{cases}\\
&=\big(\chi_v\cdot\chi_e)\boxtimes(\chi_w\cdot\chi_f)\\
&=\big(\chi_v\otimes\chi_w\big)\big(\chi_e\boxtimes\chi_f\big)\\
&=\varphi(\chi_{v\times w})\Phi_0(\chi_{e\times f})
\end{align*}
and finally
\begin{align*}
\Phi_0(\chi_{e\times f}\cdot\chi_{v\times w})&=\begin{cases}
\Phi_0(\chi_{e\times f})&\text{if }s(e\times f)=v\times w\\
\Phi_0(0)&\text{otherwise}
\end{cases}\\
&=\begin{cases}
\chi_e\boxtimes\chi_f&\text{if $\alpha^E_{\delta(f)}(s_E(e))=v$ and $s_F(f)=w$}\\
0&\text{otherwise}
\end{cases}\\
&=\begin{cases}
\chi_e\boxtimes\chi_f&\text{if $s_E(e)=\alpha^E_{\delta(f)^{-1}}(v)$ and
$s_F(f)=w$}\\
0&\text{otherwise}
\end{cases}\\
&=\chi_e\cdot\chi_{\alpha^E_{\delta(f)^{-1}}(v)}\boxtimes\chi_f\cdot\chi_w\\
&=\chi_e\cdot\alpha_{\delta(f)}(\chi_v)\boxtimes\chi_f\cdot\chi_w\\
&=(\chi_e\boxtimes\chi_f)(\chi_v\boxtimes\chi_w)\\
&=\Phi_0(\chi_{e\times f})\varphi(\chi_{v\times w})
\end{align*}
Therefore, by Lemma \ref{IsomLemma} we have that $\Phi_0$ extends
to a correspondence isomorphism $\Phi:X(E\tensor[_{\alpha^E}]\times{_\delta}F)
\to X(E)\tensor[_\gamma]\boxtimes{_\sigma} X(F)$.
\end{ex}

\begin{ex}\label{ActionCrossedIsTwistedCorr}
In this example we will show that the crossed product of a correspondence
by an action of a discrete group can be viewed as a discrete group twisted
tensor product of correspondences. Suppose $(\gamma,\alpha)$ is and action
of a discrete group $G$ on a $C^*$-correspondence $(X,A)$. We wish to show
that the reduced crossed product $X\rtimes_{\gamma,r}G$ is isomorphic
to the twisted tensor product $X\tensor[_\gamma]\boxtimes{_{\delta_G}}C_r^*(G)$
where we view $C_r^*(G)$ as a correspondence over itself. 

Recall from Corollary \ref{GeneratingSystemCrossed} that the sets 
\begin{align*}
(X\rtimes_{\gamma,r}G)_0&:=\{i_X(x)i_G^X(s):x\in X, s\in G\}\\
(A\rtimes_{\alpha,r}G)_0&:=\{i_A(a)i_G^A(s):a\in A,s\in G\}
\end{align*}
form a generating system for $X\rtimes_{\gamma,r}G$. Let $C^*_r(G)_0=\{u_s:s\in
G\}$, i.e. the image of $G$ in $C^*_r(G)$. This set is closed under multiplication,
thus we may regard $(C^*_r(G)_0,C^*_r(G)_0)$ as a generating system for the
correspondence $C^*_r(G)$. Furthermore, every element of $C^*_r(G)_0$ is
homogeneous with respect to the grading arising from $\delta_G:u_s\mapsto
u_s\otimes u_s$. Also, we may view $(X,A)$ as a generating system for itself
and then by Corollary \ref{TwistedGeneratingSystem} we see that the sets
\begin{align*}
(X\tensor[_\gamma]\boxtimes{_{\delta_G}}C^*_r(G))_0:&=\{x\boxtimes u_s:x\in
X,s\in G\}\\
(A\tensor[_\alpha]\boxtimes{_{\delta_G}}C^*_r(G))_0:&=\{a\boxtimes u_s:a\in
A,s\in G\}
\end{align*}
form a generating system for $X\tensor[_\gamma]\boxtimes{_{\delta_G}}C^*_r(G)$.
We let $\varphi$ be the isomorphism $A\rtimes_{\alpha,r}G\to
A\tensor[_\alpha]\boxtimes{_{\delta_G}}C^*_r(G)$ and define $\Phi_0:
(X\rtimes_{\gamma,r}G)_0\to (X\tensor[_\gamma]\boxtimes{_{\delta_G}}C^*_r(G))_0$
to be the map $i_X(x)i_G^X(s)\mapsto x\boxtimes u_s$. This is clearly a bijection,
but we need to establish that it preserves the inner product and left and
right actions. First, note that
\begin{align*}
\big\langle\Phi_0(i_X(x)i_G^X(s)),\Phi_0(i_X(y)i_G^X(t))\big\rangle
&=\langle x\boxtimes u_s,y\boxtimes u_t\rangle\\
&=\alpha_{s^{-1}}\big(\langle x,y\rangle\big)\boxtimes\langle u_s,u_t\rangle\\
&=\alpha_{s^{-1}}\big(\langle x,y\rangle\big)\boxtimes u_{s^{-1}t}\\
&=\varphi\Big(i_A\big(\alpha_{s^{-1}}\big(\langle x,y\rangle\big)\big)i_G^A(s^{-1}t)
\Big)\\
&=\varphi\Big(\big\langle i_X(x)i_G^X(s),i_X(y)i_G^X(t)\big\rangle\Big)
\end{align*}
also,
\begin{align*}
\Phi_0\Big(\big(i_X(x)i_G^X(s)\big)\big(i_A(a)i_G^A(t)\big)\Big)&=
\Phi_0\big(i_X\big(x\alpha_s(a)\big)i_G^X(st)\big)\\
&=x\alpha_s(a)\boxtimes u_{st}\\
&=x\alpha_s(a)\boxtimes u_su_t\\
&=(x\boxtimes u_s)(a\boxtimes u_t)\\
&=\Phi_0\big(i_X(x)i_G^X(s)\big)\varphi\big(i_A(a)i_G^A(t)\big)
\end{align*}
and finally,
\begin{align*}
\Phi_0\Big(\big(i_A(a)i_G^A(s)\big)\big(i_X(x)i_G^X(t)\big)\Big)&=
\Phi_0\big(i_X\big(a\gamma_s(x)\big)i_G^X(st)\big)\\
&=a\gamma_s(x)\boxtimes u_{st}\\
&=a\gamma_s(x)\boxtimes u_su_t\\
&=(a\boxtimes u_s)(x\boxtimes u_t)\\
&=\varphi\big(i_A(a)i_G^A(s)\big)\Phi_0\big(i_X(x)i_G^X(t)\big)
\end{align*}
Therefore, by Lemma \ref{IsomLemma}, we have that $\Phi_0$ extends to a
correspondence isomorphism $\Phi:X\rtimes_{\gamma,r}G\to X\tensor[_\gamma]
\boxtimes{_{\delta_G}}C^*_r(G)$.
\end{ex}
\begin{ex}\label{CoactionCrossedIsTwistedCorr}
In this example we will see that crossed products by coactions
on correspondences can also be viewed as twisted tensor products.
Specifically, if $(\sigma,\delta)$ is a coaction of a discrete group $G$
on a correspondence $(X,A)$, then we wish to show that $X\rtimes_\sigma G$
is isomorphic to $X\tensor[_\sigma]\boxtimes{_\lambda} c_0(G)$
where we are viewing $c_0(G)$ as a correspondence over itself. To see this,
recall from Corollary \ref{GeneratingSetCoCrossed} that the sets
\begin{align*}
(X\rtimes_\sigma G)_0&:=\{j_X(x_s)j_G^X(f):x_s\in X_s,f\in c_0(G)\}\\
(A\rtimes_\delta G)_0&:=\{j_A(a_s)j_G^A(f):a_s\in A_s,g\in c_0(G)\}
\end{align*}
form a generating system for $X\rtimes_\sigma G$.  Let $X^0:=\bigcup_{s\in
G}X_s$ and let $A^0:=\bigcup_{s\in G}A_s$. The properties of the grading
tell us that $X^0$ and $A^0$ densely span $X$ and $A$ and that $ax,xa\in
X^0$ whenever $a\in A^0$ and $x\in X^0$. Thus $(X^0,A^0)$ is a generating
system for $X$ which by definition consists of elements which are homogeneous
with respect to the gradings of $X$ and $A$. Viewing $(c_0(G),c_0(G))$ as
a generating system for the correspondence $c_0(G)$, we may apply Corollary
\ref{TwistedGeneratingSystem} and deduce that the sets
\begin{align*}
(X\tensor[_\sigma]\boxtimes{_\lambda}c_0(G))_0&:=\{x_s\boxtimes f:x_s\in
X^0,f\in c_0(G)\}\\
(A\tensor[_\alpha]\boxtimes{_\lambda}c_0(G))_0&:=\{a_s\boxtimes f:a_s\in
A^0,f\in c_0(G)\}
\end{align*}
form a generating system for the twisted tensor product $X\tensor[_\sigma]\boxtimes
{_\lambda}c_0(G)$. Let $\varphi$ be the isomorphism $A\rtimes_\delta G\to
A\tensor[_\delta]\boxtimes{_\lambda}c_0(G)$. We define $\Phi_0:(X\rtimes_\sigma
G)_0\to(X\tensor[_\sigma]\boxtimes{_\lambda}c_0(G))_0$ to be the map
$j_X(x_s)j_G^X(f)\mapsto x_s\boxtimes f$. This is clearly bijective, but
we must show that it preserves the inner product and left and right actions.
To see this, suppose $x_s\in X_s\subseteq X^0$ and $x_t\in X_t\subseteq X^0$
and note that
\begin{align*}
\big\langle\Phi_0\big((j_X(x_s)j_G^X(f)\big),\Phi_0\big((j_X(x_t)j_G^X(g)
\big)\big\rangle&=\langle x_s\boxtimes f,x_t\boxtimes g\rangle\\
&=\langle x_s,x_t\rangle\boxtimes\lambda_{s^{-1}}\big(\langle f,g\rangle\big)\\
&=\langle x_s,x_t\rangle\boxtimes\lambda_{s^{-1}}(\overline fg)\\
&=\varphi\Big(j_A\big(\langle x_s,x_t\rangle\big)j_A^G\big(\lambda_
{s^{-1}}(\overline fg)\big)\Big)\\
&=\varphi\Big(\big\langle j_X(x_s)j_G^X(f),j_X(x_t)j_G^X(g)\big\rangle\Big)
\end{align*}
furthermore, if $a_t\in A_t\subseteq A^0$ we have that
\begin{align*}
\Phi_0\Big(\big(j_X(x_s)j_G^X(f)\big)\big(j_X(a_t)j_G^X(g)\big)\Big)&=
\Phi_0\Big(j_X(x_sa_t)j_G^X\big(\lambda_t(f)g\big)\Big)\\
&=x_sa_t\boxtimes\lambda_t(f)g\\
&=(x_s\boxtimes f)(x_t\boxtimes g)\\
&=\Phi_0\big(j_X(x_s)j_G^X(f)\big)\varphi\big(j_A(a_t)j_G^A(g)\big)
\end{align*}
and
\begin{align*}
\Phi_0\Big(\big(j_A(a_t)j_G^A(f)\big)\big(j_X(x_s)j_G^X(g)\big)\Big)&=\Phi_0
\Big(j_X(a_tx_s)j_G^X\big(\lambda_s(f)g\big)\Big)\\
&=a_tx_s\boxtimes\lambda_s(f)g\\
&=(a_t\boxtimes f)(x_s\boxtimes g)\\
&=\varphi\big(j_A(a_t)j_G^A(f)\big)\Phi_0\big(j_X(x_s)j_G^X(g)\big)
\end{align*}
Therefore, by Lemma \ref{IsomLemma}, we have that $\Phi_0$ extends to
a correspondence isomorphism $\Phi:X\rtimes_\sigma G\to X\tensor[_\sigma]
\boxtimes{_\lambda}c_0(G)$.
\end{ex}
\section{Balanced Twisted Tensor Products}
Throughout this section, $G$ will be a discrete group, $Z$ will be a compact
abelian
group, $(A,G,\alpha)$, $(A,Z,\mu)$ and $(B,Z,\nu)$ will be dynamical systems
and $(B,G,\delta)$ will be a coaction such that $\mu$ commutes with $\alpha$
and $\nu$ is covariant with respect to $\delta$.
\begin{prop}
$(A\tensor[_\alpha]\boxtimes{_\delta}B,Z,\lambda)$ is a dynamical system
where $$\lambda_z=\mu_z\boxtimes\nu_{z^{-1}}$$
\end{prop}
\begin{proof}
Since $\mu$ commutes with $\alpha$ and $\nu$ is covariant with respect
to $\delta$, we know that each map $\lambda_z$ is an automorphism. To
show that $z\mapsto\lambda_z$ is a group homomorphism, note that 
\begin{align*}
\lambda_z\circ\lambda_w(a\boxtimes b)&=\lambda_z\big(\mu_w(a)
\boxtimes\nu_{w^{-1}}(b)\big)\\
&=\mu_z\circ\mu_w(a)\boxtimes\nu_{z^{-1}}\circ\nu_{w^{-1}}(b)\\
&=\mu_{zw}(a)\boxtimes\nu_{z^{-1}w^{-1}}(b)\\
&=\mu_{zw}(a)\boxtimes\nu_{w^{-1}z^{-1}}(b)\\
&=\mu_{zw}(a)\boxtimes\nu_{(zw)^{-1}}(b)\\
&=\lambda_{zw}(a\boxtimes b)
\end{align*}
\end{proof}
\begin{dfn}
We call the fixed point algebra $(A\boxtimes B)^\lambda$ of the above
action the \emph{$Z$-balanced twisted tensor product} of $A$ and $B$
and we denote it by $A\boxtimes_ZB$.
\end{dfn}
\begin{prop}
Let $a\boxtimes b\in A\boxtimes_ZB$. Then $\mu_z(a)\boxtimes b,
a\boxtimes\nu_z(b)\in A\boxtimes_ZB$ for all $z\in Z$ and, in fact,
$\mu_z(a)\boxtimes b=a\boxtimes\nu_z(b)$ for all $z\in Z$.
\end{prop}
\begin{proof}
Let $w\in Z$. Note that
\begin{align*}
\lambda_w\big(\mu_z(a)\boxtimes b\big)&=\mu_w\circ\mu_z(a)\boxtimes 
\nu_{w^{-1}}(b)\\
&=\mu_{z}\circ\mu_w(a)\boxtimes\nu_{w^{-1}}(b)\\
&=(\mu_z\boxtimes\text{id}_B)\big(\lambda_w(a\boxtimes b)\big)\\
&=(\mu_z\boxtimes\text{id}_B)(a\boxtimes b)\\
&=\mu_z(a)\boxtimes b
\end{align*}
Thus $\mu_z(a)\boxtimes b$ is fixed under the action of $\lambda$ and
is therefore an element of $A\boxtimes_Z B$.
Showing that $a\boxtimes\nu_z(b)\in A\boxtimes_Z B$ is similar. To show that
these are actually equivalent, notice that
\begin{align*}
\mu_z(a)\boxtimes b&=\lambda_{z^{-1}}\big(\mu_z(a)\boxtimes b\big)\\
&=\mu_{z^{-1}}\circ\mu_z(a)\boxtimes\nu_z(b)\\
&=a\boxtimes\nu_z(b)
\end{align*}
\end{proof}
\begin{prop}\label{GradingOfProduct}
Keeping the above conventions, let $\delta^\mu$ and $\delta^\nu$ be
the dual coactions. Since $Z$ is compact, $\widehat Z$ will be discrete
and thus $\delta^\mu$ and $\delta^\nu$ give gradings $\{A_\chi\}_{\chi\in
\widehat Z}$ and $\{B_\chi\}_{\chi\in\widehat Z}$ of $A$ and $B$. For each
$\chi\in\widehat Z$ let 
$$S_\chi:=\{a\boxtimes b:a\in A_\chi,b\in B_\chi\}$$
and let $S:=\bigcup_{\chi\in\widehat Z}S_\chi$. Then $A\boxtimes_ZB=
\overline{\text{\emph{span}}}(S)$.
\begin{proof}
First we will show that $\overline{\text{span}}(S)\subseteq A\boxtimes_ZB$.
Let $a\boxtimes b\in S$. Then $a\boxtimes b\in S_\chi$
for some $\chi\in\widehat Z$. Thus for all $z\in Z$ we have that
\begin{align*}
\lambda_z(a\boxtimes b)&=\mu_z(a)\boxtimes \nu_{z^{-1}}(b)\\
&=\chi(z)a\boxtimes \chi(z^{-1})b\\
&=\chi(z)\chi(z^{-1})(a\boxtimes b)\\
&=a\boxtimes b
\end{align*}
thus $S\subseteq A\boxtimes_ZB$ and therefore $\overline{\text{span}}(S)
\subseteq A\boxtimes_ZB$.

To show the reverse inclusion, let $c\in A\boxtimes_Z B$. Then $c\approx\sum
a_i\boxtimes b_i$. Since the subspaces $\{A_\chi\}_{\chi\in\widehat Z}$ 
span $A$ and the subspaces $\{B_\chi\}_{\chi\in\widehat Z}$ span $B$, we
may assume without loss of generality that there are $\chi_i,\chi'_i\in
\widehat Z$ such that $a_i\in A_{\chi_i}$ and $b_i\in B_{\chi'_i}$ for
each $i$. Let $\Upsilon:A\boxtimes B\to A\boxtimes_{\mathbb T}B$ be the 
conditional expectation $d\mapsto \int_Z\lambda_z(d)dz$. Since $c$ is assumed
to be in $A\boxtimes_\mathbb T B$, we have that $c=\Upsilon(c)$. Thus, using
the continuity and linearity of $\Upsilon$ we have the following
\begin{align*}
c&=\Upsilon(c)\\&=\int_Z\lambda_z(c)dz\\
&\approx\int_Z\lambda_z\Big(\sum_i a_{\chi_i}\boxtimes b_{{\chi'}_i}\Big)dz\\
&=\sum_i\int_Z\lambda_z\big(a_{\chi_i}\boxtimes b_{{\chi'}_i})dz\\
&=\sum_i\int_Z\mu_z(a_{\chi_i})\boxtimes\nu_{z^{-1}}(b_{ {\chi'}_i})dz\\
&=\sum_i\int_Z\big(\chi_i(z)a_{\chi_i}\big)\boxtimes\big({\chi'}_i(z^{-1})b_{{\chi'}_i}\big)dz\\
&=\sum_i \big(a_{\chi_i}\boxtimes b_{ {\chi'}_i}\big)\int_Z\chi_i(z){\chi'}_i(z^{-1})dz\\
\end{align*}
But the integral $\int_Z\chi_i(z)\chi'_i(z^{-1})dz$ is equal to $1$ if $\chi_i=\chi'_i$
and zero otherwise (this is a consequence of the Peter-Weyl theorem). 
Thus $c$ can be approximated by a sum of elements
from $S$:$$c\approx\sum_i a_{\chi_i}\boxtimes b_{\chi_i}$$
therefore $A\boxtimes_\mathbb T B\subseteq\overline{\text{span}}(S)$
so $A\boxtimes_\mathbb T B=\overline{\text{span}}(S)$.
\end{proof}
\end{prop}
\begin{prop}\label{BalancedAction}
Let $\{S_\chi\}_{\chi\in\widehat Z}$ be as above. This defines a coaction
of $\widehat Z$ on $A\boxtimes_ZB$ and thus also an action
$\gamma$ of $Z$  on $A\boxtimes_ZB$. We have that 
$\gamma=\mu\boxtimes\emph{id}_B=\emph{id}_A\boxtimes\nu$.
\end{prop}
\begin{proof}
Let $c\in A\boxtimes_\mathbb TB$. By the previous proposition we can approximate
$c\approx \sum_i a_{\chi_i}\boxtimes b_{\chi_i}$. Then, by continuity and
linearity
\begin{align*}
\gamma_z(c)&\approx\sum_i\gamma_z\big(a_{\chi_i}\boxtimes b_{\chi_i}\big)\\
&=\sum_i\chi_i(z)\big(a_{\chi_i}\boxtimes b_{\chi_i}\big)\\
&=\sum_i\big(\chi_i(z)a_{\chi_i}\big)\boxtimes b_{\chi_i}\\
&=\sum_i\mu_z(a_{\chi_i})\boxtimes b_{\chi_i}
\end{align*}
Thus $\gamma=\mu\boxtimes \text{id}_B$. Showing that $\gamma=\text{id}_A
\boxtimes\nu$ is similar.
\end{proof}
\begin{lem}\label{Full}
Suppose $\{A_\chi\}_{\chi\in\widehat Z}$ and $\{B_\chi\}_{\chi\in\widehat
Z}$ are saturated gradings of $A$ and $B$, that is $A_\chi A_\omega=A_{\chi\omega}$.
Then $\{S_\chi\}_{\chi\in\widehat Z}$ (as described above) is saturated.
\end{lem}
\begin{proof}
We already have that $S_\chi S_\omega\subseteq S_{\chi\omega}$ for all
$\chi,\omega\in\widehat Z$ so it will suffice to show the reverse inclusion.
Let $a\boxtimes b\in S_{\chi\omega}$. Then $a\in A_{\chi\omega}$ and 
$b\in B_{\chi\omega}$. Since  the gradings of $A$ and $B$ are saturated,
we have that $a\approx\sum_i a_{\chi,i}a_{\omega,i}$ and
$b\approx\sum_j b_{\chi,j}b_{\omega,j}$ with $a_{\chi,i}\in A_\chi$,
$a_{\omega,i}\in A_\omega$, $b_{\chi,j}\in B_\chi$ and $b_{\omega,j}\in B_\omega$.
Thus 
\begin{align*}
a\boxtimes b&=\sum_{i,j}a_{\chi,i}a_{\omega,i}\boxtimes b_{\chi,j}b_{\omega,j}\\
&=\sum_{i,j}\big(a_{\chi,i}\boxtimes b_{\chi,i}\big)\big(a_{\omega,i}\boxtimes
b_{\omega,j}\big)\\
&\subseteq S_\chi\boxtimes S_\omega
\end{align*}
\end{proof}
\section{Ideal Compatibility}
Before we state our main result, we need to define two technical
conditions involving the Katsura ideals. We call these conditions
Katsura nondegeneracy and ideal compatibility. These conditions are
basically the same as those presented in the author's first paper 
\cite{Morgan}, only the definition of ideal compatibility must be 
modified to allow twisted tensor products instead of ordinary ones.
\begin{dfn}\label{KatsuraNonDeg}
Let $(X,A)$ be a correspondence and let $J_X$ be the Katsura ideal of $A$
(recall that the Katsura ideal is the ideal $J_X=\phi^{-1}\big(\mathcal K(X)
\big)\cap\big(ker(\phi)\big)^{\perp}$). We say that $X$ is \emph{Katsura
nondegenerate} if $J_X\cdot X=X$.
\end{dfn}
\begin{ex}
Let $X$ be a correspondence over a $C^*$-algebra $A$ such that the
left action is injective and implemented by compacts. In this case
we have that $J_X=A$. Thus:
\begin{align*}
X\cdot J_X&=X\cdot A\\
&=X
\end{align*}
\end{ex}
\begin{dfn}\label{ProperSource}
Recall that a vertex in a directed graph is called a \emph{source}
if it receives no edges. We will call such a vertex a \emph{proper
source} if it emits at least one edge.
\end{dfn}
\begin{prop}\label{GraphKatNonDeg}
Let $E$ be a directed graph. Then $X(E)$ is Katsura nondegenerate if
and only if $E$ has no proper sources and no infinite receiver emits an edge.
\end{prop}
\begin{proof}
Suppose there is $v\in E^0$ such that $|r^{-1}(v)|=\infty$ and
$|s^{-1}(v)|>0$. Then for every $f\in J_X$ we have $f(v)=0$. Thus
for any $g\in C_c(E^1)$, $f\in J_X$, and $e\in s^{-1}(v)$, 
we have $(g\cdot f)(e)
=g(e)f\big(s(e)\big)=g(e)f(v)=0$. Thus $h(e)=0$ for all $h\in C_c(E^1)
\cdot J_X$ and, taking the limit, $x(e)=0$ for all $x\in X\cdot J_X$. 
Thus $\delta_e\notin X\cdot J_X$ since $\delta_e(e)=1\neq
0$ but $\delta_e\in X$. Therefore $X\neq X\cdot J_X$, i.e. $X$ is not
Katsura nondegenerate.

Similarly, suppose that $E$ has a proper source $v$. Then, since $|r^{-1}(v)|=0$
we must have $f(v)=0$ for all $f\in J_X$. Then for any $g\in C_c(E^1)$ and
$e\in s^{-1}(v)$ we have that $(g\cdot f)(e)=g(e)f(v)=0$ for $f\in J_X$.
Thus by similar reasoning as above we have that $x(e)=0$ for all $x\in
X\cdot J_X$ and so $\delta_e\notin X\cdot J_X$ but $\delta_e\in X$ and
we can again conclude that $X\neq X\cdot J_X$ so $X$ is not Katsura
nondegenerate.

On the other hand, suppose $E$ has no proper sources and no infinite receiver
in $E$ emits an edge.
Let $e\in E^1$ and let $v=s(e)$. Then $|r^{-1}(v)|<\infty$ 
and $|r^{-1}(v)|>0$ by assumption, so function in $J_X$ can be 
supported on $v$. In particular, $\delta_v\in J_X$. Since $\delta_e
\cdot\delta_v=\delta_e$ we know that $\delta_e\in X\cdot J_X$.
Since $e$ was arbitrary, we have that all such characteristic functions
are contained in $X\cdot J_X$. But these functions densely span $C_c(E^1)$
and thus densely span $X$,  so we have that $X\subseteq X\cdot J_X$ and therefore
$X=X\cdot J_X$ so $X$ is Katsura nondegenerate.
\end{proof}
We now turn our attention to ideal compatibility:
\begin{dfn}\label{IdealCompatible}
Given two correspondences $(X,A)$ and $(Y,B)$, an action $(\gamma,\alpha)$
of a discrete group $G$ on $X$ and a coaction $(\sigma,\delta)$ of $G$ on
$Y$, we say that
$X$ and $Y$ are \emph{ideal compatible} with respect to $\gamma$ and $\sigma$
if $J_X\tensor[_\alpha]\boxtimes{_\delta} J_Y=J_{X\tensor[_\gamma]\boxtimes{
_\sigma}Y}$ where $J_X,J_Y$ and $J_{X\tensor[_\gamma]\boxtimes{
_\sigma}Y}$ are the Katsura ideals of $X,Y$, and $X\tensor[_\gamma]\boxtimes{
_\sigma}Y$ respectively.
\end{dfn}
Once again, this condition will be met in the case of injective left actions
implemented by compacts:
\begin{prop}\label{GraphIdealComp}
Suppose we have correspondences $(X,A)$ and $(Y,B)$, an action $(\gamma,\alpha)$
of a discrete group $G$ on $X$ and a coaction $(\sigma,\delta)$ of $G$ on
$Y$. If the left actions of $A$ on $X$ and $B$ on $Y$ are injective and
implemented by compacts, then $X$ and $Y$ are ideal compatible. 
\end{prop}
\begin{proof}
In this case the Katsura ideals are the whole algebras: $J_X=A$, $J_Y=B$
and $J_{X\tensor[_\gamma]\boxtimes{_\sigma}Y}=A\tensor[_\alpha]\boxtimes
{_\delta}B$. Thus the equation 
$J_X\tensor[_\alpha]\boxtimes{_\delta} J_Y=J_{X\tensor[_\gamma]\boxtimes{
_\sigma}Y}$ becomes trivial.
\end{proof}
In Example 8.13 of \cite{GraphAlgebraBook},
it is shown that if $E$ is a discrete graph, then
$$J_{X(E)}=\overline{span}\{\chi_v:0<|r^{-1}(v)|<\infty\}$$
where $X(E)$ is the associated correspondence and $\chi_v\in c_0(E^0)$
denotes the characteristic function of the vertex $v\in E^0$. With this in
mind, we give the following proposition:
\begin{prop}\label{GraphCompatible}
Let $E$ and $F$ be discrete graphs and let $\big(X=X(E),A=c_0(E^0)\big)$
and $\big(Y=X(F),B=c_0(F^0)\big)$ be
the associated correspondences. Suppose $\alpha^E$ is an action of a 
discrete group $G$ on $E$
and let $(\gamma,\alpha)$ be the associated action on $(X,A)$. Let $\delta:
F^1\to G$ be a labeling of the edges of $F$ and let $(\sigma,\iota)$ be the
associated coaction on $(Y,B)$. Then $X$ and $Y$ are ideal compatible with
respect to $\gamma$ and $\sigma$.
\end{prop}
\begin{proof} 
Recall that $X\tensor[_\gamma]\boxtimes{_\sigma}Y=X(E\tensor[_{\alpha^E}]
\times{_\delta} F)$ where $E\tensor[_{\alpha^E}]\times{_\delta} F$ is the
graph defined in Example \ref{GraphProductCorr}. Thus
$$J_{X\tensor[_\gamma]\boxtimes{_\sigma} Y}=\overline{span}\{\chi_{v\times
w}:0<|r_{E\tensor[_{\alpha^E}]\times{_\delta}F}^{-1}(v\times w)|<\infty\}$$
By definition, $r_{E\tensor[_{\alpha^E}]\times{_\delta}F}=r_E\times r_F$
so $r_{E\tensor[_{\alpha^E}]\times{_\delta}F}^{-1}
(v\times w)=r^{-1}_E(v)\times r^{-1}_F(w)$ and thus $|r^{-1}_{E\tensor
[_{\alpha^E}]\times{_\delta}F}(v\times
w)|=|r_E^{-1}(v)|\cdot|r^{-1}_F(w)|$. But $0<|r_E^{-1}(v)|\cdot|r^{-1}_F(w)|<
\infty$ if and only if $0<|r_E^{-1}(v)|<\infty$ and $0<|r^{-1}_F(w)|<\infty$.
Thus we have that
$$J_{X\tensor[_\gamma]\boxtimes{_\sigma}Y}=\overline{span}\{\chi_{v\times
w}:0<|r_E^{-1}(v)|,|r^{-1}_F(w)|<\infty\}$$
If we identify $c_0(E^0\times
F^0)$ with $c_0(E^0)\otimes c_0(F^0)$ in the standard way, we see that 
$\chi_{v\times x}=\chi_v\otimes\chi_w$ (recall from Example \ref{GraphProductCorr}
that $\iota$ is the trivial coaction so we have that $A\tensor[_\alpha]\boxtimes{_\iota}
B\cong A\otimes B$). Thus
\begin{align*}
J_{X\tensor[_\gamma]\boxtimes{_\sigma}Y}&=\overline{span}\{\chi_v\otimes
\chi_w:0<|r_E^{-1}(v)|,|r^{-1}_F(w)|<\infty\}\\
&=\overline{span}\{f\otimes g:f\in J_X,g\in J_Y\}\\
&=J_X\otimes J_Y
\end{align*}
Therefore, $X$ and $Y$ are ideal-compatible.
\end{proof}

\section{Main Result}
We will now state our main theorem, although we delay the proof until
later in this section. Throughout this section $X$ and $Y$ will be 
correspondences over $C^*$-algebras $A$ and $B$, $(\sigma,\delta)$ 
will be 
a coaction of a discrete group $G$ on $Y$ and $(\gamma,\alpha)$ will be an
action of $G$ on $X$. We will denote the induced action on $\mathcal O_X$
by $\gamma'$ and the induced coaction on $\mathcal O_Y$ by $\sigma'$.
To simplify the notation, we will make the following definitions:
$X\boxtimes Y:=X\tensor[_\gamma]\boxtimes{_\sigma}Y$, $A\boxtimes B
:=A\tensor[_\alpha]\boxtimes{_\delta}B$, and $\mathcal O_X\boxtimes_\mathbb
T \mathcal O_Y:=\mathcal O_X\tensor[_{\gamma'}]\boxtimes{_{\sigma',\mathbb
T}}\mathcal O_Y$.
\begin{thm}\label{MainResult}
Suppose $X$ and $Y$ are full, ideal compatible, 
and Katsura nondegenerate (see the previous section).
Then $\mathcal O_{X\tensor[_\gamma]\boxtimes{_\sigma} Y}\cong\mathcal O_X
\tensor[_{\gamma'}]\boxtimes{_{\sigma',\mathbb T}}\mathcal O_Y$.
\end{thm}
Before we can prove this result, we will need the following lemmas:
\begin{lem}\label{ToeplitzRepn}
Suppose $(\pi_X,\psi_X)$ and $(\pi_Y,\psi_Y)$ are Toeplitz representations
of $X$ and $Y$ in $C^*$-algebras $C$ and $D$ and let $\gamma'$ and $\sigma'$
be the induced action and coaction on $C$ and $D$. Let $\psi:=\psi_X\otimes
\psi_Y$ and let $\pi:=\pi_X\otimes\pi_Y$. Then $(\pi,\psi)$ is a Toeplitz
representation of $X\tensor[_\gamma]\boxtimes{_\sigma}Y$ in $C\tensor[_{\gamma'}]
\boxtimes{_{\sigma'}}D$.
\end{lem}
\begin{proof}
Let $y\in Y_s$ for some $s\in G$,
\begin{align*}
\psi\big((x\boxtimes y)(a\boxtimes b)\big)&=\psi\big(x\alpha_s(a)\boxtimes
yb\big)\\
&=\psi_X\big(x\alpha_s(a)\big)\boxtimes\psi_Y(yb)\\
&=\psi_X(x)\pi_X\big(\alpha_s(a)\big)\boxtimes\psi_Y(y)\pi_Y(b)\\
&=\psi_X(x)\gamma'_s\big(\pi_X(a)\big)\boxtimes\psi_Y(y)\pi_Y(b)\\
&=\big(\psi_X(x)\boxtimes\psi_Y(y)\big)\big(\pi_X(a)\boxtimes\pi_Y(b)\big)\\
&=\psi(x\boxtimes y)\pi(a\boxtimes b)
\end{align*}
A similar argument with $b\in B_s$ shows that
$$\psi\big((a\boxtimes b)(x\boxtimes y)\big)=\pi(a\boxtimes b)\psi(x\boxtimes
y)$$
Further, 
\begin{align*}
\psi(x\boxtimes y)^*\psi(x'\boxtimes y')&=\Big(\psi_X(x)
\boxtimes \psi_Y(y)\Big)^*\Big(\psi_X(x')\boxtimes \psi_Y(y')\Big)\\
&=\Big(\gamma'_{s^{-1}}\big(\psi_X(x)^*\big)\boxtimes\psi_Y(y)^*\Big)
\Big(\psi_X(x')\boxtimes\psi_Y(y')\Big)\\
&=\gamma'_{s^{-1}}\big(\psi_X(x)^*\big)\gamma'_{s^{-1}}\big(\psi_X(x')\big)
\boxtimes\psi_Y(y^*)\psi_Y(y')\\
&=\gamma'_{s^{-1}}\big(\psi_X(x)^*\psi_X(x')\big)\boxtimes\psi_Y(y)^*\psi_Y(y')\\
&=\gamma'_{s^{-1}}\Big(\pi_X\big(\langle x,x'\rangle_A\big)\Big)\boxtimes
\pi_Y\big(\langle y,y'\rangle_B\big)\\
&=\pi_X\Big(\gamma_{s^{-1}}\big(\langle x,x'\rangle_A\big)\Big)\boxtimes
\pi_Y\big(\langle y,y'\rangle_B\big)\\
&=\pi\Big(\gamma_{s^{-1}}\big(\langle x,x'\rangle_A\big)\boxtimes\langle
y,y'\rangle_B\Big)\\
&=\pi\big(\langle x\boxtimes y,x'\boxtimes y'\rangle_{A\boxtimes B}\big)
\end{align*}
These equalities can be extended extend linearly and continuously (by the 
linearity and continuity of the maps $\psi_X,\pi_X,\psi_Y,\pi_Y,\psi$, and
$\pi$) to all of $X\boxtimes Y$ 
and $A\boxtimes B$ thus $(\pi,\psi)$ is a Toeplitz representation.
\end{proof}
\begin{lem}
If $(\pi,\psi)$ is the Toeplitz representation in Lemma \ref{ToeplitzRepn},
then:
$$\psi^{(1)}\big(k(S\boxtimes T)\big)=\psi_X^{(1)}(S)\boxtimes\psi_Y^{(1)}
(T)$$
for all $S\in\mathcal K(X)$ and $T\in\mathcal K(Y)$.
\end{lem}
\begin{proof}
Without loss of generality we may assume that $S=\Theta_{x,x'}$ and 
$T=\Theta_{y,y}$ with $y\in Y_s$ and $y'\in Y_t$. We have:
\begin{align*}
\psi^{(1)}\big(k(\Theta_{x,x'}\boxtimes\Theta_{y,y'})\big)&=
\psi^{(1)}\big(\Theta_{x\boxtimes y,\gamma_{ts^{-1}}(x')\boxtimes y'}
\big)\\
&=\psi(x\boxtimes y)\psi\big(\gamma_{ts^-1}(x')\boxtimes y'\big)^*\\
&=\Big(\psi_X(x)\boxtimes\psi_Y(y)\Big)\Big(\psi_X\big(\gamma_{ts^{-1}}(x')\big)
\boxtimes\psi_Y(y')\Big)^*\\
&=\Big(\psi_X(x)\boxtimes\psi_Y(y)\Big)\Big(\gamma'_{t^{-1}}\big(
\psi_X\big(\gamma_{ts^{-1}}(x')\big)^*\big)\boxtimes\psi_Y(y')^*\Big)\\
&=\Big(\psi_X(x)\boxtimes\psi_Y(y)\Big)\Big(\gamma'_{s^{-1}}\big(
\psi_X(x')^*\big)\boxtimes\psi_Y(y')^*\Big)\\
&=\big(\psi_X(x)\psi_X(x')^*\big)\boxtimes\big(\psi_Y(y)\psi_Y(x')^*\big)\\
&=\psi_X^{(1)}\big(\Theta_{x,x'}\big)\boxtimes\psi_Y^{(1)}\big(\Theta_{y,y'}\big)
\end{align*}
\end{proof}
\begin{lem}\label{CPCovariant}
Suppose that $X$ and $Y$ are ideal compatible.
If $(\pi_X,\psi_X)$ and $(\pi_Y,\psi_Y)$ are Cuntz-Pimsner covariant
then so is $(\pi,\psi)$.
\end{lem}
\begin{proof}
Let $\phi_X$, $\phi_Y$, and $\phi$ be the left action maps on $X$, $Y$
and $X\tensor[_\gamma]\boxtimes{_\sigma}Y$ respectively. Note that 
from the definition of $X\tensor[_\gamma]\boxtimes{_\sigma}Y$
we have that $\phi=j\circ\big(\phi_X\boxtimes\phi_Y\big)$ and thus,in
particular $\phi|_{J_{X\boxtimes Y}}=k\circ\big(\phi_X|_{J_X}\boxtimes
\phi_Y|_{J_Y})$\(\). 
Let $c\in J_{X\boxtimes Y}$. Since $J_{X\boxtimes Y}=J_X\boxtimes J_Y$
we may approximate $c\approx\sum a_i\boxtimes b_i$ for some
$a_i\in J_X$ and $b_i\in J_Y$. We have
\begin{align*}
\psi^{(1)}\big(\phi(c)\big)&\approx\sum\psi^{(1)}
\big(\phi(a_i\boxtimes b_i)\big)\\
&=\sum\psi^{(1)}\big(k\big(\phi_X(a_i)\boxtimes \phi_Y(b_i)\big)\\
&=\sum\psi^{(1)}_X\big(\phi_X(a_i)\big)\boxtimes\psi^{(1)}_Y\big
(\phi_Y(b_i)\big)\\
&=\sum\pi_X(a_i)\boxtimes\pi_Y(b_i)\\
&=\sum\pi(a_i\boxtimes b_i)\\
&\approx\pi(c)
\end{align*}
thus $(\pi,\psi)$ is Cuntz-Pimsner covariant.
\end{proof}
We are now ready to prove our main theorem.
\begin{proof}{(of Theorem \ref{MainResult})}
First, we will construct a $*$-homomorphism $F:\mathcal O_{X\boxtimes Y}
\to\mathcal O_X\boxtimes\mathcal O_Y$. Let $(k_X,k_A)$ and $(k_Y,k_B)$
be the standard Cuntz-Pimsner covariant representations of $X$ and $Y$ 
in $\mathcal O_X$ and $\mathcal O_Y$ and let $\psi:=k_X\boxtimes k_Y$ and
$\pi:=k_A\boxtimes k_B$. Then by Lemma \ref{CPCovariant}, $(\psi,\pi)$ is
a Cuntz-Pimsner covariant representation of $X\boxtimes Y$ in
$\mathcal O_X\boxtimes\mathcal O_Y$. Thus, by the universal property
there is a unique homomorphism $F:\mathcal O_{X\boxtimes Y}\to\mathcal O_X\boxtimes
\mathcal O_Y$ such that
$$(\psi,\pi)=(F\circ k_{X\boxtimes Y},F\circ k_{A\boxtimes B})$$
The gauge actions $\Gamma_X$ and $\Gamma_Y$ on $\mathcal O_X$ and
$\mathcal O_Y$ give rise to $\mathbb Z$ gradings $\{\mathcal O_X^n\}_
{n\in\mathbb Z}$ and $\{\mathcal O_Y^n\}_{n\in\mathbb Z}$. By Proposition
\ref{GradingOfProduct} the sets 
$$S_n:=\{x\boxtimes y:x\in\mathcal O_X^n,y\in\mathcal O_Y^n\}$$
give a $\mathbb Z$-grading of $\mathcal O_X\boxtimes_\mathbb T\mathcal O_Y$.
Notice that we have the following:
\begin{align}
&F\big(k_{A\boxtimes B}(a\boxtimes b)\big)=\pi(a\boxtimes b)=
(k_A\boxtimes k_B)(a\boxtimes b)\in S_0\label{AtimesBinS_0}\\
&F\big(k_{X\boxtimes Y}(x\boxtimes y)\big)=\psi(x\boxtimes y)=(k_X\boxtimes
k_Y)(x\boxtimes y)\subseteq S_1\label{XtimesYinS_1}
\end{align}
For $a\in A,b\in B,x\in X, y\in Y$.
Since the image of $F$ is generated by $F\big(k_{A\boxtimes B}
(A\boxtimes B\big)$
and $F\big(k_{X\boxtimes Y}(X\boxtimes Y)\big)$
we see that the image of $F$ lies inside $\mathcal O_X\boxtimes_
\mathbb T\mathcal O_Y$. Furthermore, note that the grading $\{S_n\}_
{n\in\mathbb Z}$ gives rise to an action $\Gamma$ of $\mathbb T$
such that $\Gamma_z(s)=z^ns$ for $s\in S_n$. In particular, by 
(\ref{AtimesBinS_0}) 
we have that $\Gamma_z\big(\pi(c)\big)=\pi(c)$ for all $c\in A\boxtimes B$
and by (\ref{XtimesYinS_1}) we have that 
$\Gamma_z\big(\psi(w)\big)=z\psi(w)$ for all
$w\in X\boxtimes Y$. Thus $\Gamma$ is a gauge action for $(\psi,\pi)$.
Since $k_A$, $k_B$, $k_X$, and $k_Y$ are all injective, $\psi$ and $\pi$
will be injective. By the Gauge Invariant Uniqueness Theorem, we have that
$F$ is injective.

It remains to show that $F$ is surjective onto $\mathcal O_X\boxtimes_\mathbb
T\mathcal O_Y$. Recall that, since $X$ and $Y$ are full, so are the gradings
$\{\mathcal O_X^n\}_{n\in\mathbb Z}$ and $\{\mathcal O_Y^n\}_{n\in\mathbb
Z}$ and thus by Lemma \ref{Full} so is the grading $\{S_n\}_{n\in\mathbb
Z}$. Thus $\mathcal O_X\boxtimes_\mathbb T\mathcal O_Y$ is generated by the
elements of $S_1$ (see Lemma 3.7 of \cite{Morgan}). 
Therefore, to show that $F$
is surjective, it suffices to show that $S_1$ is in the image of $F$.

Recall that $\mathcal O_X^1$ is densely spanned
by elements of the form:
$k_X^{n+1}(x) k^n_X(x')^*$. Similarly,
$\mathcal O_Y^1$ is densely spanned by elements of the form
$k_Y^{n+1}(y)k^n_Y(y')^*$. 
Thus the elements
\begin{align}
k^{n+1}_X(x)k^n_X(x')^*\boxtimes k_Y^{m+1}(y) k^m_Y(y')^*\label{GeneratorForS_1}
\end{align}
are dense in $S_1$.
Fix such an element. We may assume without loss of
generality that $y$ is a tensor product of homogeneous elements
(for if not, it can be approximated by a sum of such products).
In other words, we may assume that $y=y_1\otimes\cdots\otimes y_{m+1}$ 
where each $y_i$ 
homogeneous with respect to the $G$-grading of $Y$. Let us say that
$y_i\in Y_{s_i}$ for each $i$. Let $s:=s_1\cdots s_{m+1}$. Then 
$$k^{m+1}_Y(y)=k_Y(y_1)\cdots k_Y(y_{m+1})\in\mathcal O_Y^s$$
Similarly, we can assume that $y'=y'_1\otimes\cdots\otimes y'_m$
with $y'_i\in Y_{t_i}$ and then $k_Y^m(y')\in\mathcal O_Y^t$
where $t:=t_1\cdots t_m$.

Let us assume $n\leq m$. Then $m=n+l$ for some $l\geq0$.
Then we can factor 
$y=y^{(1)}\otimes y^{(2)}$ and $y'=y'^{(1)}\otimes y'^{(2)}$.
Since $X$ is Katsura nondegenerate, we can factor
$x=x_0a$ and $x'=x'_0a'$ with $x_0,x'_0\in X$ and $a,a'\in J_X$.
With this in mind, we can factor (\ref{GeneratorForS_1}) as follows:
\begin{align*}
k&^{n+1}_X(x)k^n_X(x')^*\boxtimes k_Y^{m+1}(y) k^m_Y(y')^*\\
&=\big(k_X^{n+1}(x)\boxtimes k_Y^{m+1}(y)\big)\Big({\gamma'}_{s^{-1}}
\big(k_X^n(x')\big)^*\boxtimes k_Y^m(y')^*\Big)\\
&=\big(k_X^{n+1}(x_0)k_A(a)\boxtimes k_Y^{n+1}(y^{(1)})k_Y^l(y^{(2)})\big)
\\&\quad\Big({\gamma'}_{s^{-1}}\big(k_A(a')\big)^*{\gamma'}_{s^{-1}}
\big(k_X^n({x'}_0)\big)^*\boxtimes k_Y^l({y'}^{(2)})^*k_Y^n({y'}^{(1)})^*\Big)\\
&=\Big(k_X^{n+1}(x_0)\boxtimes k_Y^{n+1}(y^{(1)})\Big)\Big({\gamma'}_{s_1^{-1}}
\big(k_A(a)\big)\boxtimes k_Y^l(y^{(2)})\Big)\\
&\quad\Big({\gamma'}_{s^{-1}}\big(k_A({a'})\big)^*
\boxtimes k_Y^l({y'}^{(2)})^*\Big)\Big({\gamma'}_{t_2s^{-1}}\big(k_X^n({x'}_0)\big)^*
\boxtimes k_Y^n({y'}^{(1)})^*\Big)\\
&=\Big(k_X^{n+1}(x_0)\boxtimes k_Y^{n+1}(y^{(1)})\Big)\Big({\gamma'}_{s_1^{-1}}
\big(k_A(a)\big){\gamma'}_{s_2s^{-1}}\big(k_A({a'})\big)^*\boxtimes k_Y^l(y^{(2)})
k_Y^l({y'}^{(2)})^*\Big)\\
&\quad\Big({\gamma'}_{t_2s^{-1}}\big(k_X^n({x'}_0)\big)^*
\boxtimes k_Y^n({y'}^{(1)})^*\Big)\\
&=\Big(k_X^{n+1}(x_0)\boxtimes k_Y^{n+1}(y^{(1)})\Big)\Big(
k_A\big(\alpha_{s_1^{-1}}(a)\alpha_{s_2s^{-1}}({a'})^*\big)\boxtimes k_Y^l(y^{(2)})
k_Y^l({y'}^{(2)})^*\Big)\\
&\quad\Big({\gamma'}_{t_1t_2s^{-1}}\big(k_X^n({x'}_0)\big)
\boxtimes k_Y^n({y'}^{(1)})\Big)^*\\
&=\Big(k_X^{n+1}(x_0)\boxtimes k_Y^{n+1}(y^{(1)})\Big)\Big(
k_X^{(1)}\circ\phi_X\big(\alpha_{s_1^{-1}}(a)\alpha_{s_2s^{-1}}({a'})^*\big)
\boxtimes k_Y^{(1)}\big(\Theta_{y^{(2)},{y'}^{(2)}}\big)
\Big)\\
&\quad\Big(k_X^n\big(\alpha_{ts^{-1}}({x'}_0)\big)
\boxtimes k_Y^n({y'}^{(1)})\Big)^*\\
&=\psi^{n+1}\Big(x_0\boxtimes y^{(1)}\Big)\psi^{(1)}\Big(\phi_X\big(\alpha_
{s_1^{-1}}(a)\alpha_{s_2s^{-1}}({a'})^*\big)\boxtimes\Theta_{y^{(2)},{y'}^{(2)}}\Big)\\
&\quad\psi^n\Big(\alpha_{ts^{-1}}({x_0}')\boxtimes {y'}^{(1)}\Big)^*
\end{align*}

Since all three factors in the final line are in the algebra generated
by $(\psi,\pi)$, we know that $k^{n+1}_X(x)k^n_X(x')^*\boxtimes k_Y^{m+1}(y)
k^m_Y(y')^*$ must be in the image of $F$. If $n>m$ we can factor
$x=x^{(1)}\otimes x^{(2)}$, $x'=x'^{(1)}\otimes x'^{(2)}$ $y=y_0a$, and
$y'=y_0'a'$ and apply a similar argument. So the image of $F$ generates $S_1$
and thus $F$ is surjective onto $\mathcal O_X\tensor[_{\gamma'}]\boxtimes{_{\sigma'}}
_\mathbb\mathcal O_Y$.
Hence we have established that $\mathcal O_{X\tensor[_\gamma]\boxtimes{_\sigma}Y}
\cong\mathcal O_X\tensor[_{\gamma'}]\boxtimes{_{\sigma',\mathbb T}}\mathcal O_Y$
\end{proof}

\section{Examples}
In this section, we will apply our main theorem to some of the examples 
of twisted tensor products of correspondences we discussed earlier.
\begin{ex}
Let $X$ and $Y$ be $\mathbb Z_2$-graded correspondences over 
$\mathbb Z_2$-graded $C^*$-algebras $A$ and $B$. As in Example 
\ref{GradedPorductCorr},
let $(\gamma,\alpha)$ be the action of $\mathbb Z_2$ on $(X,A)$
associated to the grading on $X$ and let $(\sigma,\delta)$ be the
coaction of $\mathbb Z_2$ on $(Y,B)$ associated to the grading of
$Y$. Let $\gamma'$ be the action of $\mathbb Z_2$ on $\mathcal O_X$ 
induced by $(\gamma,\alpha)$ and let $\sigma'$ be the coaction of
$\mathbb Z_2$ on $\mathcal O_Y$ induced by $(\sigma,\delta)$. Suppose
$X$ and $Y$ are Katsura nondegenerate and ideal compatible with respect 
to $\sigma$ and $\gamma$. Then
by our main result together with Example \ref{GradedPorductCorr}, we see
that 
$\mathcal O_{X\widehat\otimes Y}\cong\mathcal O_{X\boxtimes Y}\cong
\mathcal O_X\tensor[_{\gamma'}]\boxtimes{_{\sigma',\mathbb T}}\mathcal O_Y$.

Note that $\gamma'$ and $\sigma'$ give rise to $\mathbb Z_2$-gradings
on $\mathcal O_X$ and $\mathcal O_Y$ so it makes sense to speak of the
$\mathbb Z_2$-graded tensor product $\mathcal O_X\widehat\otimes\mathcal
O_Y$. Let $\varphi:\mathcal O_X\widehat\otimes\mathcal O_Y\to\mathcal O_X
\boxtimes\mathcal O_Y$ be the map $w\widehat\otimes z\mapsto w\boxtimes
z$. Note that  
\begin{align*}
\varphi\big((w_1\widehat\otimes z_1)(w_2\widehat\otimes z_2)\big)&=
(-1)^{\partial z_1\partial w_2}\varphi(w_1w_2\widehat\otimes z_1z_2)\\
&=(-1)^{\partial z_1\partial w_2}(w_1w_2\boxtimes z_1z_2)\\
&=w_1\gamma'_{\partial z_1}(w_2)\boxtimes z_1z_2\\
&=(w_1\boxtimes z_1)(w_2\boxtimes z_2)\\
&=\varphi(w_1\widehat\otimes z_1)\varphi(w_2\widehat\otimes z_2)
\end{align*}
and also
\begin{align*}
\varphi\big((w\widehat\otimes z)^*\big)&=(-1)^{\partial w\partial z}\varphi(
w^*\widehat\otimes z^*)\\
&=(-1)^{\partial w\partial z}(w^*\boxtimes z^*)\\
&=(-1)^{\partial w^*\partial z^*}(w^*\boxtimes z^*)\\
&=\gamma_{\partial z^*}(w^*)\boxtimes z^*\\
&=(w\boxtimes z)^*\\
&=\varphi(w\widehat\otimes z)^*
\end{align*}
therefore $\varphi$ is a $*$-homomorphism. Further, $\varphi$ takes a 
densely spanning set (the elementary tensors $w\widehat\otimes z$) in
$\mathcal O_X\widehat\otimes \mathcal O_Y$ bijectively onto a densely 
spanning set
(the elementary tensors $w\boxtimes z$) in $\mathcal O_X\tensor[_{\gamma'}]\boxtimes
{_{\sigma'}}\mathcal O_Y$ 
so $\varphi$ is an isomorphism. Thus we have that
$$\mathcal O_{X\widehat\otimes Y}\cong\mathcal O_X\widehat\otimes_\mathbb
T\mathcal O_Y$$
\end{ex}
\begin{ex}
In this example we will show that the graph algebra of the product 
graph constructed in Example \ref{GraphProductCorr} is a twisted tensor 
product of the graph algebras of the underlying graphs. Suppose $E$ and $F$
are directed graphs, $\alpha^E$ an action of a discrete group $G$ on $E$
and let $\delta$ be a labelling of the edges of $F$ by elements of $G$.
Let $(\gamma,\alpha)$ and $(\sigma,\iota)$ be the associated action
and coaction on the graph correspondences $(X,A)=\big(X(E),c_0(E^0)\big)$
and $(Y,B)=\big(X(F),c_0(F^0)\big)$. Suppose further, that $E$ and $F$
have no sinks, no proper sources (see Definition \ref{ProperSource}),
and that no infinite receiver in $E$ or $F$ emits an edge.
Then we have
that the graph correspondences $(X,A)$ and $(Y,B)$ are full, and by 
Propositions \ref{GraphKatNonDeg}
and \ref{GraphIdealComp} they are Katsura nondegenerate and ideal compatible
with respect to $\gamma$ and $\sigma$. Thus we may apply our main result.
Let $\gamma'$ be the induced action on $\mathcal O_X\cong C^*(E)$ and 
$\sigma'$ be the induced coaction on $\mathcal O_Y\cong C^*(F)$. We have
$$C^*(E\tensor[_{\alpha^E}]\times{_\delta}F)\cong\mathcal O_{X\tensor[_\gamma]
\boxtimes{_\sigma}Y}\cong\mathcal O_X\tensor[_{\gamma'}]\boxtimes{_{\sigma',
\mathbb T}}\mathcal O_Y\cong C^*(E)\tensor[_{\gamma'}]\boxtimes{_{\sigma',
\mathbb T}}C^*(F)$$
\end{ex}

Before we see how our main result applies to crossed products, we will need
a few lemmas.
\begin{lem}
Let $(A,G,\alpha)$ be a dynamical system, $(B,G,\delta)$ a coaction
and $C$ a $C^*$-algebra. There exists an isomorphism 
$$\sigma_{23}:A\tensor[_\alpha]\boxtimes{_{\Sigma_{23}(\delta\otimes
\emph{id}_C)}}(B\otimes C)
\to(A\otimes C)\tensor[_{\alpha\otimes\iota}]\boxtimes{_\delta}B$$
\end{lem}
\begin{proof}
By definition, $$A\tensor[_\alpha]\boxtimes{_{\Sigma_{23}(\delta\otimes
\emph{id}_C)}}(B\otimes C)=i_A(A)\cdot i_{B\otimes C}(B\otimes C)
\subseteq A\otimes B\otimes C\otimes \mathcal B\big(L^2(G)\big)$$
where $i_A=\delta^\alpha_{1m}$ and $$i_{B\otimes C}=\circ\big(\Sigma_{23}
(\delta\otimes\text{id}_C)\big)_{23\lambda}$$. Similarly, 
$$(A\otimes C)\tensor[_{\alpha\otimes\iota}]\boxtimes{_\delta}B=
i_{A\otimes C}(A\otimes C)\cdot i_B(B)\subseteq A\otimes C\otimes B
\otimes \mathcal B\big(L^2(G)\big)$$
where $i_{A\otimes C}=m_4\circ\delta^{\alpha\otimes\iota}_{12m}$ and $i_B=
\lambda_4\circ\delta
_{3\lambda}$. Since $\Sigma_{23}:A\otimes B\otimes C\otimes\mathcal B\big(L^2(G)\big)
\to A\otimes C\otimes B\otimes\mathcal B\big(L^2(G)\big)$ is an isomorphism,
we can let $\sigma_{23}$ be the restriction of $\Sigma_{23}$ to the subalgebra
$A\tensor[_\alpha]\boxtimes{_{\Sigma_{23}(\delta\otimes
\text{id}_C)}}(B\otimes C)$. From there it suffices to show that the image
of $\sigma_{23}$ is equal to $(A\otimes C)\tensor[_{\alpha\otimes\iota}]
\boxtimes{_\delta}B$. $A\tensor[_\alpha]\boxtimes{_{\Sigma_{23}(\delta\otimes
\emph{id}_C)}}(B\otimes C)$ is densely spanned by elements of the form
$i_A(a)i_{B\otimes C}(b\otimes c)$ and $(A\otimes C)\tensor[_{\alpha\otimes\iota}]
\boxtimes{_\delta}B$ is densely spanned by elements of the form
$i_{A\otimes C}(a\otimes c)i_B(b)$. Since $\sigma_{23}$ is a restriction
of the linear norm preserving map $\Sigma_{23}$, it too is linear and norm
preserving, thus to show that the image of $\sigma_{23}$ is 
$(A\otimes C)\tensor[_{\alpha\otimes\iota}]\boxtimes{_\delta}B$, it suffices
to show that $\sigma_{23}\big(i_A(a)i_{B\otimes C}(b\otimes c)\big)=
i_{A\otimes C}(a\otimes c)i_B(b)$. To see this, note that
\begin{align*}
\sigma_{23}\big(i_A(a)i_{B\otimes C}(b\otimes c)\big)&=\Sigma_{23}
\Big(\delta^{\alpha}(a)_{1m}\big(\Sigma_{23}(\delta\otimes\text{id}_C)(b\otimes
c)\big)_{23\lambda}\Big)\\
&=\Sigma_{23}\Big(\delta^\alpha(a)_{1m}\big(\Sigma_{23}\big(\delta(b)\otimes
c\big)\big)_{23\lambda}\Big)\\
&=\delta^\alpha(a)_{1m}\Sigma_{23}\big(c_3\delta(b)_{2\lambda}\big)\\
&=\delta^\alpha(a)_{1m}c_2\delta(b)_{3\lambda}\\
&=\Big(\Sigma_{23}\big(\delta^{\alpha}(a)\otimes c\big)\Big)_{12m}\delta(b)
_{3\lambda}\\
&=\delta^{\alpha\otimes\iota}(a\otimes c)_{12m}\delta(b)_{3\lambda}\\
&=i_{A\otimes C}(a\otimes c)i_B(b)
\end{align*}
thus $\sigma_{23}:A\tensor[_\alpha]\boxtimes{_{\Sigma_{23}(\delta\otimes
\emph{id}_C)}}(B\otimes C)
\to(A\otimes C)\tensor[_{\alpha\otimes\iota}]\boxtimes{_\delta}B$ is an isomorphism.
\end{proof}
\begin{lem}\label{CPalgofXtimesB}
Let $(\gamma,\alpha)$ be an action of a discrete group $G$ on
a correspondence $(X,A)$ and let $(B,G,\delta)$ be a coaction.
If $J_{X\tensor[_\gamma]\boxtimes{_\delta} B}=J_X\tensor[_\alpha]\boxtimes
{_\delta} B$ and $J_X$ is Katsura-nondegenerate,
then
$$\mathcal O_{X\tensor[_\gamma]\boxtimes{_\delta} B}\cong\mathcal O_X\tensor
[_{\gamma'}]\boxtimes{_\delta} B$$
where $B$ is viewed as a correspondence over itself in the left-hand-side
and $\gamma'$ is the induced action on $\mathcal O_X$.
\end{lem}
\begin{proof}
In this context, our main result gives us 
$$\mathcal O_{X\tensor[_\gamma]\boxtimes{_\delta} B}\cong\mathcal O_X
\tensor[_{gamma'}]\boxtimes{_{delta,\mathbb T}}\mathcal O_B$$
Recall that $\mathcal O_B\cong B\otimes C(\mathbb T)$ and the gauge action
corresponds to $\iota\otimes\lambda$ where $\iota$ is the trivial action
and $\lambda$ is left translation. Using the isomorphism between $C(\mathbb
T)$ and $C^*(\mathbb Z)$, we see that 
\begin{align*}
\mathcal O_X\boxtimes_\mathbb T \mathcal O_B&\cong\mathcal O_X\boxtimes_
\mathbb T\big(B\otimes C^*(\mathbb Z)\big)\\
&=\overline{\text{span}}\big\{\mathcal O_X^n\boxtimes\big(B\otimes C^*(\mathbb
Z)^n\big)\big\}\\
&=\overline{\text{span}}\big\{x\boxtimes(b\otimes u_n)\in\mathcal O_X
\boxtimes\big(B\otimes C^*(\mathbb Z)\big):x\in\mathcal O_X^n\big\}
\end{align*}
where $u_n$ is the unitary in $C^*(\mathbb Z)$ associated to $n\in
\mathbb Z$. Applying the isomorphism $\sigma_{23}$ from the previous lemma,
we see that the above is isomorphic to 
\begin{align*}
&\overline{\text{span}}\big\{(x\otimes u_n)\boxtimes b\in\big(\mathcal O_X\otimes
C^*(\mathbb Z)\big)\boxtimes B:x\in \mathcal O_X^n\big\}\\
=&\overline{\text{span}}\big\{x\otimes u_n\in\mathcal O_X\otimes C^*(\mathbb
Z):x\in\mathcal O_X^n\big\}\boxtimes B
\end{align*}
The grading $\{\mathcal O_X^n\}_{n\in\mathbb Z}$ corresponds to a coaction
$\varepsilon$ such that $\varepsilon(x)=x\otimes u_n$ for all $x\in\mathcal
O_X^n$. Continuing from above, we have:
\begin{align*}
&\overline{\text{span}}\big\{x\otimes u_n\in\mathcal O_X\otimes C^*(\mathbb
Z):x\in\mathcal O_X^n\big\}\boxtimes B\\
=&\overline{\text{span}}\big\{\varepsilon(x):x\in\mathcal O_X^n\big\}\boxtimes B\\
=&\varepsilon\big(\mathcal O_X\big)\boxtimes B\\
\cong&\mathcal O_X\boxtimes B
\end{align*}
\end{proof}
\begin{lem}\label{CPalgofBtimesX}
Let $(\sigma,\delta)$ be a coaction of a discrete group $G$
on a full correspondence $(X,A)$. Let $\alpha$ be an action of $G$ on a $C^*$-algebra
$B$. Suppose $J_{X\tensor[_\sigma]\boxtimes{_\alpha}B}=J_X\tensor[_\delta]\boxtimes
{_\alpha} B$ and
$J_X$ is Katsura-nondegenerate.
Then
$$\mathcal O_{X\tensor[_\sigma]\boxtimes{_\alpha} B}\cong\mathcal O_X\tensor
[_{\sigma'}]\boxtimes{_\alpha}B$$
where $\sigma'$ is the induced coaction on $\mathcal O_X$
\end{lem}
\begin{proof}
The proof is similar to the proof of Lemma \ref{CPalgofXtimesB}.
\end{proof}
\begin{ex}
Let $(X,A)$ be a correspondence and let $(\gamma,\alpha)$ be an action
of a discrete group $G$ on $A$. Suppose further, that $J_{X\rtimes_{\gamma,r}
G}=J_X\rtimes_{\alpha,r}G$.
Recall that Example \ref{ActionCrossedIsTwistedCorr} showed that 
$X\rtimes_{\gamma,r} G\cong X\tensor[_\gamma]\boxtimes{_{\delta_G}} C_r^*(G)$.
Thus we have$$\mathcal O_{X\rtimes
_{\gamma,r}G}\cong\mathcal O_{X\tensor[_\gamma]\boxtimes{_{\delta_G}} C_r^*(G)}$$
Since $A\rtimes_{\alpha,r} G\cong A\tensor[_\alpha]\boxtimes{_{\delta_G}}
C_r^*(G)$, the condition that $J_{X\rtimes_{\gamma,r}
G}=J_X\rtimes_{\alpha,r}G$ can be restated as $J_{X\tensor[_\gamma]
\boxtimes{_{\delta_G}} C_r^*(G)}
=J_X\tensor[_\alpha]\boxtimes{_{\delta_G}} C_r^*(G)$.
Thus the correspondences $X$ and $C_r^*(G)$
to be ideal compatible.
Applying Lemma \ref{CPalgofXtimesB}, we have that $\mathcal O_{X
\tensor[_\gamma]\boxtimes{_{\delta_G}} C_r^*(G)}\cong\mathcal
O_X\tensor[_{\gamma'}]\boxtimes{_{\delta_G}} C_r^*(G)$ and this, in turn,
is isomorphic to $\mathcal O_X\rtimes_{\gamma',r} G$. Thus we have:
$$\mathcal O_{X\rtimes_{\gamma,r} G}\cong\mathcal O_X\rtimes_{\gamma',r}
G$$
\end{ex}
\begin{ex}
Suppose $(\sigma,\delta)$ is
a coaction of a discrete group $G$ on a correspondence $(X,A)$. Then
$X\rtimes_\sigma G\cong X\tensor[_\sigma]\boxtimes{_\lambda}c_0(G)$ by
Proposition \ref{CoactionCrossedIsTwistedCorr}. If $J_{X\tensor[_\sigma]
\boxtimes{_\lambda}c_0(G)}=J_X\tensor[_\delta]\boxtimes{_\lambda}c_0(G)$
(or in other terms $J_{X\rtimes_\sigma G}=J_X\rtimes_\delta G$) then Lemma
\ref{CPalgofBtimesX} tells us that
$$\mathcal O_{X\tensor[_\sigma]\boxtimes{_\lambda}c_0(G)}\cong\mathcal
O_X\tensor[_{\sigma'}]\boxtimes{_\lambda}c_0(G)$$ which we can restate
as $$\mathcal O_{X\rtimes_\sigma G}\cong\mathcal O_X\rtimes_{\sigma'}G$$
where $\sigma'$ is the induced coaction on $\mathcal O_X$.
\end{ex}
These last two examples are not new. In \cite{HaoNg} it was shown that
$\mathcal O_{X\rtimes_\gamma G}\cong\mathcal O_X\rtimes_{\gamma'}G$ for
any locally compact amenable group, and in \cite{Quigg3} it is shown that
$\mathcal O_{X\rtimes_{\gamma,r} G}\cong\mathcal O_X\rtimes_{\gamma',r}G$
for any locally compact group provided that $J_{X\rtimes_{\gamma,r}G}=
J_X\rtimes_{\alpha,r}G$. In \cite{Quigg2} it was shown that 
$\mathcal O_{X\rtimes_\sigma G}\cong\mathcal O_X\rtimes_{\sigma'}G$
for full (not reduced) coactions $\sigma$ provided certain technical
conditions involving the Katsura ideal are satisfied. Even though our
main result does not recover these results in their full generality,
it seems reasonable to hope that the main result of this paper might be 
extended to arbitrary locally compact groups and perhaps even quantum groups.
In this case we would have a single framework in which to describe all
results of this type.

\end{document}